\documentclass[12pt,a4paper]{amsart} 
\usepackage{amsmath,amsfonts,amssymb,amsthm,amscd,slashed}
\usepackage{epsfig}
\pdfoutput=1

\usepackage{tikz}

\usepackage{xypic}

\usepackage{hyperref}
\usepackage{JK}                

\newtheorem{thm}{Theorem}

\newtheorem{lma}[thm]{Lemma} \newtheorem{prop}[thm]{Proposition}
\newtheorem{defn}[thm]{Definition}

\newtheorem{rem}[thm]{Remark}

\newtheorem{ass}{Assumption}

\def\1{\mathbb{I}}

\def\C{\mathbb{C}}
\DeclareMathOperator{\Cl}{Cl}

\DeclareMathOperator{\dom}{Dom}

\def\E{\mathcal{E}}

\DeclareMathOperator{\End}{End}

\def\Hom{\mathrm{Hom}}

\def\S{\mathbb{S}}

\def\T{\mathbb{T}}

\def\tilde{\widetilde}

\def\U{\mathcal{U}}

\def\cX{\mathcal{X}}

\title[Factorization of Dirac operators on toric NC manifolds]{Factorization of Dirac operators on toric noncommutative manifolds}
\author{Jens Kaad and Walter D. van Suijlekom}
\address{Department of Mathematics and Computer Science, Syddansk Universitet, Campusvej 55, 5230, Odense M, Denmark}
\email{jenskaad@hotmail.com}
\address{Institute for Mathematics, Astrophysics and Particle Physics, Radboud University Nijmegen, Heyendaalseweg 135, 6525 AJ Nijmegen, The Netherlands}
\email{waltervs@math.ru.nl}

\date{\today}

\dedicatory{Dedicated to Alain Connes on the occasion of his 70th birthday}

\subjclass[2010]{58B34;19K35,53C27} 
\keywords{Toric noncommutative manifolds, Connes-Landi sphere, Unbounded Kasparov modules, Half-closed chains, Dirac operators, Spin$^c$-manifolds, Unbounded KK-theory, Unbounded Kasparov product}

\begin{document}
\bibliographystyle{amsalpha-lmp}

\begin{abstract}
We factorize the Dirac operator on the Connes--Landi 4-sphere in unbounded KK-theory. We show that a family of Dirac operators along the orbits of the torus action defines an unbounded Kasparov module, while the Dirac operator on the principal orbit space ---an open quadrant in the 2-sphere--- defines a half-closed chain. We show that the tensor sum of these two operators coincides up to unitary equivalence with the Dirac operator on the Connes--Landi sphere and prove that this tensor sum is an unbounded representative of the internal Kasparov product in bivariant K-theory.

We also generalize our results to Dirac operators on all toric noncommutative manifolds subject to a condition on the principal stratum. We find that there is a curvature term that arises as an obstruction for having a tensor sum decomposition in unbounded KK-theory. This curvature term can however not be detected at the level of bounded KK-theory.
\end{abstract}

\maketitle

\section{Introduction}
At the beginning of this millennium a new class of noncommutative spin manifolds was discovered by Connes and Landi \cite{CL01}. They are based at the topological level on Rieffel's deformation along a torus action \cite{Rie93a}. The propotypical example is given by the $\theta$-deformed four-sphere $\S^4_\theta$ ---also referred to nowadays as the Connes--Landi sphere--- and, in fact, in \cite{CL01} an explicit formula was given for the Dirac operator on this noncommutative space. This was elaborated on in \cite{CD02}. In the work of the second author as a PhD student of  Landi a key role was played by this explicit example of a noncommutative spin manifold in the analysis of noncommutative instanton moduli spaces ({\it cf.} \cite{LS04,LS06}). 

It is the purpose of the present paper to dissect that very Dirac operator on the Connes--Landi sphere by writing it as a tensor sum of a vertical family of Dirac operators and a horizontal Dirac operator. The vertical family of Dirac operators differentiates along the orbits of the torus action on $\S^4$ defining the deformation, the horizontal data corresponds to the torus invariant part. More precisely, we will write the Dirac operator $D_{\S^4_\theta}$ (up to unitary equivalence and on a core) as
\begin{equation}
\label{eq:tensor-sum-S4}
D_{\S^4_\theta}= D_V \otimes 1 + 1 \otimes_\nabla D_{Q_0^2}
\end{equation}
where $D_V$ is a suitable family of Dirac operators and $D_{Q_0^2}$ is the (closed symmetric) Dirac operator on the principal orbit space denoted $Q_0^2$. We show that our factorization result fits well into the framework of noncommutative geometry \cite{C94} by proving that the tensor sum represents the interior Kasparov product of the corresponding classes in bivariant K-theory.

In fact, we will show that such a factorization result holds for Dirac operators on all toric noncommutative manifolds $M_\theta$ that were described by Connes and Landi \cite{CL01}, subject to a condition on the principal stratum. Thus, we identify a vertical family of Dirac operators $D_V$ that allows us to write the Dirac operator $D_{M_\theta}$ as a tensor sum (analogous to Equation \eqref{eq:tensor-sum-S4}) with the Dirac operator on the principal orbit space. We then show that this tensor sum is an unbounded representative of the interior Kasparov product of the corresponding classes in bivariant K-theory. An obstruction that appears in this tensor sum decomposition and which is only visible at the level of unbounded KK-theory is given by the curvature of the fibration.

The present work can be considered as part of the search for a suitable geometric description of the internal gauge degrees of freedom that arise from a noncommutative structure. Indeed, any noncommutative unital $*$-algebra $A$ gives rise to a non-abelian group $\U(A)$ of invertible (unitary) elements in $A$. This has given rise to many applications in physics, such as to Yang--Mills theories \cite{CC96,CC97,BS10} and to the Standard Model of elementary particles 
\cite{CCM07,CM07,Cac11,BD13,CCS13b}. In all of these examples the elements in $\U(A)$ are realized as automorphisms of a principal bundle, in perfect agreement with the usual description of gauge theories. It is important to remark, however, that in the noncommutative approach the gauge group and the gauge fields are defined along very general lines \cite{C96}, valid for any spectral triple on a $C^*$-algebra $A$, but that in this generality the geometric picture is less clear. Nevertheless, the work on instanton moduli spaces on toric noncommutative manifolds  \cite{BL09a,BS10,BLS11} demonstrate that in noncommutative geometry internal gauge parameters beg for a description on the same geometric footing as the usual gauge degrees of freedom. 

These results motivated previous work on factorizations of the Dirac operator on the $\theta$-deformed Hopf fibration \cite{BMS13}, and, more generally, on toric noncommutative manifolds \cite{FR15b,CM16}. However, in these examples it was of crucial importance that the base manifold appears as the orbit space of a {\em free} torus action. This was in contrast with the topological bundle picture we derived for the internal gauge degrees of freedom, valid for any real spectral triple \cite{Sui14b}. As a special class of examples, it was shown that the deformation of a Riemannian spin manifold along a torus action can be described at the $C^*$-algebraic ({\it i.e.} topological) level by a $C^*$-bundle on the (possibly singular) orbit space. This applies in particular to the Connes--Landi sphere: one can construct a $C^*$-bundle on the singular orbit space $\S^4/\T^2$ for the torus action on the 4-sphere whose space of continuous sections is isomorphic to the $C^*$-algebra describing the noncommutative 4-sphere.

We now take the second step and push this bundle description for the Connes--Landi sphere ---in fact, for toric noncommutative manifolds in general--- to the geometric level by showing that also the Dirac operator can be decomposed into a vertical operator acting on continuous sections of a Hilbert bundle, and a horizontal Dirac operator on the (principal) orbit space (Theorems \ref{thm:tensor-sum-toric} and \ref{thm:fact-toric} below). 
This is in line with the recent paper \cite{KS17a} in which we deal with factorizations of Dirac operators on almost-regular fibrations. 


\subsection*{Acknowledgements}
%

We gratefully acknowledge the Syddansk Universitet Odense and the Radboud University Nijmegen for their financial support in facilitating this collaboration. 

During the initial stages of this research project the first author was supported by the Radboud excellence fellowship.

The first author was partially supported by the DFF-Research Project 2 ``Automorphisms and Invariants of Operator Algebras'', no. 7014-00145B and by the Villum Foundation (grant 7423).

The second author was partially supported by NWO under VIDI-grant 016.133.326.

\section{Dirac operator on the Connes-Landi sphere}
The Connes-Landi sphere arises as a deformation of the spin geometry of the round four-sphere $\S^4$. It is based on the introduction of the noncommutative torus along the action on $\S^4$ of a 2-torus. Before we give a precise definition of the spin geometry of the noncommutative sphere, let us start by collecting some basic facts about the classical geometry of the round 4-sphere.

\subsection{The round 4-sphere}
We parametrize the round 4-sphere $\S^4$ by toroidal coordinates: $0\leq \theta_1,\theta_2<2 \pi$, $0 \leq  \varphi\leq \pi/2$ and $-\pi/2 \leq \psi \leq \pi/2$:
\begin{equation}
\label{eq:spherical-S4}
a= e^{i \theta_1} \cos \varphi \cos \psi; \qquad b =e^{i \theta_2} \sin \varphi \cos \psi; \qquad x = \sin \psi
\end{equation}
so that indeed $|a|^2 +|b|^2+x^2 =1$. In these coordinates, the metric is given by 
\begin{align}
\label{eq:metric-S4}
g_{\S^4} &= \cos^2 \varphi \cos^2 \psi ~ d \theta_1^2 + \sin^2 \varphi \cos^2 \psi ~ d \theta_2^2+\cos^2 \psi ~ d \varphi^2 + d\psi^2.
\end{align}
There is an action of $\T^2$ given by translation in the coordinates $\theta_1$ and $\theta_2$. Even though this action is not free there is principal stratum, which is a dense open subset of $\S^4$ and we denote it by $\S^4_0$. Note that
\[
\S^4_0 = \{ a \neq 0, b \neq 0\} 
= \S^4 \setminus\left( (\S^4)^{ \{ 0 \} \times \T} \cup (\S^4)^{\T \ti \{0\}} \right)
\]
where the subsets 
\[
\begin{split}
& (\S^4)^{\T \ti \{0\}} = \{ (a,b,x) \in \S^4 \mid a = 0\} \q \Tex{and} \\
& (\S^4)^{ \{ 0 \} \times \T} =  \{(a,b,x) \in \S^4 \mid b = 0\}
\end{split}
\]
are embedded 2-spheres, and hence both compact embedded submanifolds of $\S^4$ of codimension two. This will turn out to be useful for the analysis later on.

%
%

The quotient space $\S^4/\T^2$ is denoted by $Q^2$; we write $\pi: \S^4 \to Q^2$ for the projection map. More precisely, 
\[
Q^2 = \left\{ (\cos \varphi \cos \psi, \sin \varphi \cos \psi, \sin \psi)\mid 0 \leq \varphi\leq \pi/2, -\pi/2 \leq \psi \leq \pi/2 \right\}
\]
which is a closed quadrant of the two-sphere $\S^2$.  
We also write 
\[
Q^2_0= \left\{ (\cos \varphi \cos \psi, \sin \varphi \cos \psi, \sin \psi)\mid 0 < \varphi < \pi/2, -\pi/2 < \psi < \pi/2 \right\}
\]
for the interior of $Q^2$, an open quadrant in $\S^2$. Note that $Q_0^2$ is isomorphic to an open square on which the vector fields $\pa/\pa \varphi$ and $\pa/\pa \psi$ make sense globally. The induced map 
\[
\pi_0 = \pi|_{\S^4_0}: \S^4_0 \to Q^2_0
\] 
is then a trivial $\T^2$-principal fiber bundle, hence a submersion. This becomes a Riemannian submersion if we equip $\S^4_0$ with the induced metric from $\S^4$ (cf. Equation \eqref{eq:metric-S4}) and $Q^2_0$ with the induced metric from $\S^2$, i.e.,
\begin{equation}
\label{eq:metric-S2}
g_{Q_0^2} = \cos^2 \psi ~ d \varphi^2 + d\psi^2.
\end{equation}

\begin{figure}
\includegraphics[scale=.5]{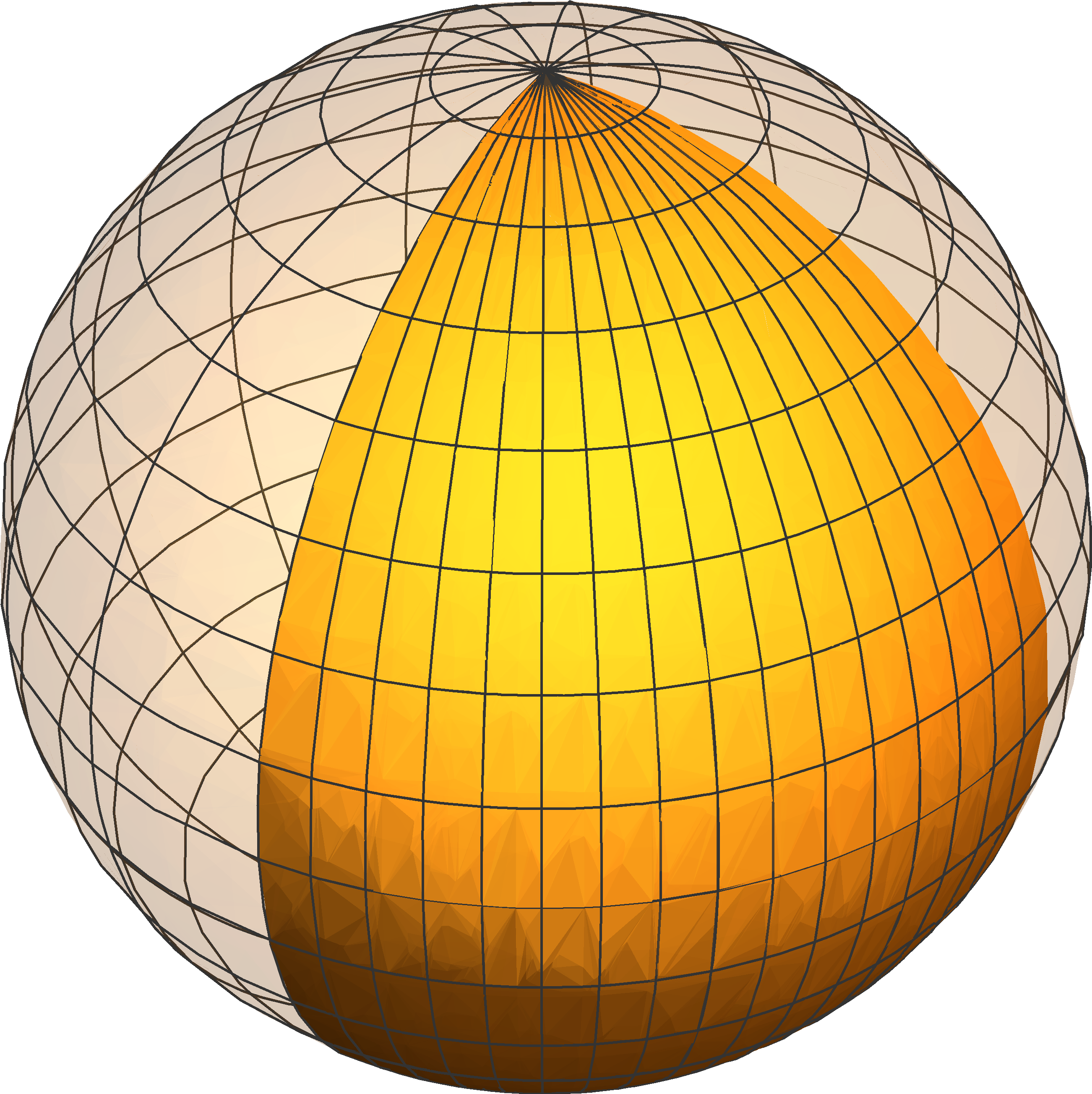}
\caption{The base space of the Riemannian submersion $\pi_0 : \S^4_0 \to Q_0^2$ is an open quadrant in the 2-sphere; the fibers are given by 2-tori. }
\label{fig:quadrant}
\end{figure}

\begin{rem}
The above context of a Riemannian submersion $\pi_0 : \S^4_0 \to Q_0^2$ allows us to speak of vertical and horizontal vector fields on $\S^4_0$, spanned by the respective pairs $\pa/\pa\theta_1,\pa/\pa\theta_2$ and $\pa/\pa\varphi, \pa/\pa \psi$. Accordingly, we may compute the second fundamental form $S$, the curvature $\Omega$ and the mean curvature $k$ of the Riemannian submersion $\pi_0: \S^4_0 \to Q^2_0$ (cf. \cite[Definition 3]{KS16} and \cite[Section 10.1]{BGV92}). The non-vanishing components of the second fundamental form are 
\begin{gather*}
S\left(\frac{\pa}{\pa\theta_1} , \frac{\pa}{\pa \theta_1},\frac{\pa}{\pa \varphi}\right) = - \sin \varphi \cos \varphi \cos^2 \psi  
\\ S\left(\frac{\pa}{\pa\theta_2} , \frac{\pa}{\pa \theta_2},\frac{\pa}{\pa \varphi}\right) =\sin \varphi \cos \varphi \cos^2 \psi
\\
S\left(\frac{\pa}{\pa\theta_1} , \frac{\pa}{\pa \theta_1},\frac{\pa}{\pa \psi}\right) = - \cos^2 \varphi \sin \psi \cos \psi 
\\
S\left(\frac{\pa}{\pa\theta_2} , \frac{\pa}{\pa \theta_2},\frac{\pa}{\pa \psi}\right) = - \sin^2 \varphi \sin \psi \cos \psi.
\end{gather*}
so that the mean curvature becomes 
\begin{equation}
\label{eq:mean-curv-S4} 
k\left(\frac{\pa}{\pa \varphi}\right) = \cot \varphi - \tan \varphi;\qquad
k\left(\frac{\pa}{\pa \psi}\right) = -2 \tan \psi.
\end{equation}
The curvature $\Omega$ is trivial because the horizontal vector fields close under the Lie bracket. 
\end{rem}

\subsection{The noncommutative $n$-torus}\label{ss:nc-torus}
We recall some basic concepts on the noncommutative $n$-torus \cite{C80,Rie81}. Thus, let $\te \in M_n(\rr)$ be a fixed skew-symmetric matrix for some $n \in \{2,3,\ldots \}$. The noncommutative $n$-torus is built on the universal unital $*$-algebra $\C[\T_\te^n]$ generated by $n$ unitaries $U_1,U_2,\ldots,U_n$ that satisfy the relations
\begin{equation}
\label{eq:nc-torus}
U_m U_l = e^{2 \pi i \cd \theta_{ml}} \cd U_l U_m \q l,m \in \{1,\ldots,n\}.
\end{equation}
The unital $*$-algebra $\C[\T_\te^n]$ carries a faithful tracial linear map $\tau : \C[\T_\te^n] \to \C$ defined by
\[
\tau : \sum_{k \in \zz^n} \la_k U^k \mapsto \la_{(0,0,\ldots,0)},
\]
where $U^k = U^{k_1}_1 \clc U^{k_n}_n$ for all $k \in \zz^n$. We thus have the Hilbert space $L^2(\T_\te^n)$ defined as the completion of $\C[\T_\te^n]$ with respect to the (pre)-inner product $\inn{\cd,\cd} : \C[\T_\te^n] \ti \C[\T_\te^n] \to \C$ defined by
\[
\inn{x,y} = \tau(x^* y) \q x,y \in \C[\T_\te^n].
\]
The multiplication operation in the unital $*$-algebra $\C[\T_\te^n]$ then induces a faithful action of $\C[\T_\te^n]$ on $L^2(\T_\te^n)$ and the topology of the noncommutative $n$-torus is described by the unital $C^*$-algebra $C(\T_\te^n)$ generated by the unitaries $U_1, \ldots, U_n$ considered as bounded operators on $L^2(\T_\te^n)$. 
%
The $C^*$-algebra $C(\T_\te^n)$ carries a strongly continuous action of the $n$-torus
\[
\si : \T^n \ti C(\T_\te^n) \to C(\T_\te^n)
\]
defined on generators by $\sigma_{t}(U_l) = t_l \cd U_l$. We will also denote the restriction of $\sigma$ to the coordinate algebra $\C[\T_\te^n]$ by the same symbol $\sigma$. The action $\sigma$ also extends to a strongly continuous unitary representation of $\T^n$ on $L^2(\T_\te^n)$, still denoted by $\sigma$.  

The faithful tracial linear map $\tau : \C[\T_\te^n] \to \C$ extends to a $\T^n$-invariant faithful tracial state on the unital $C^*$-algebra $C(\T_\te^n)$.

For each $j \in \{1,2,\ldots,n\}$, we have a $*$-derivation
\[
\de_j : \C[\T_\te^n] \to \C[\T_\te^n] \q \de_j : \sum_{k \in \zz^n} \la_k U^k \mapsto \sum_{k \in \zz^n} k_j \cd \la_k U^k,
\]
and together with the $C^*$-norm $\| \cd \|$ on $C(\T_\te^n)$ this gives rise to the fundamental system of seminorms $\{ \| \cd \|_m \}_{m \in \nn_0^n}$ on $\C[\T_\te^n]$ defined by
\[
\| \cd \|_m : x \mapsto \sum_{i_1 = 0}^{m_1} \sum_{i_2 = 0}^{m_2} \ldots \sum_{i_n = 0}^{m_n} \| \de_1^{i_1} \de_2^{i_2} \clc \de_n^{i_n}(x) \|.
\]
The unital Fr\'echet $*$-algebra $C^\infty(\T_\te^n)$ of smooth functions on the noncommutative $n$-torus is then defined as the completion of the unital $*$-algebra $\C[\T_\te^n]$ with respect to the fundamental system of seminorms $\{ \| \cd \|_m \}_{m \in \nn_0^n}$.
%

\subsection{The Connes-Landi sphere}
The noncommutative 4-sphere introduced by Connes and Landi in \cite{CL01} arises essentially by inserting the noncommuting unitaries $U_1$ and $U_2$ in the place of the toric coordinate functions $e^{i \theta_1}$ and $e^{i \theta_2}$, respectively. Let us give a precise definition, following basically \cite{CD02}. We start with the pullback $\alpha$ of the torus action on $\S^4$ to continuous functions $C(\S^4)$: 
\[
\alpha_{t}(f)(p) = f( t^{-1} \cdot p) \q  t \in \T^2, p \in \S^4.
\]
It is clear from this definition that $\alpha: \T^2 \ti C(\S^4) \to C(\S^4)$ defines a strongly continuous action of the $2$-torus. Moreover, this action restricts to the smooth functions $C^\infty(\S^4)$ and we will also denote the corresponding action by $\alpha: \T^2 \ti C^\infty(\S^4) \to C^\infty(\S^4)$.

For each $k \in \zz^2$ we let $C^\infty(\S^4)_k \subset C^\infty(\S^4)$ denote the spectral subspace
\[
C^\infty(\S^4)_k = \{ f \in C^\infty(\S^4) \mid \al_t(f) = t^k \cd f \Tex{ for all } t \in \T^2 \}.
\]
Similarly, using that our torus action restricts to the principal stratum, we have the spectral subspaces $C^\infty_c(\S^4_0)_k \subset C^\infty_c(\S^4_0)$. Let us fix the deformation parameter $\te = \te_{12} \in \rr$.
 
\begin{defn}
The coordinate algebra for the noncommutative 4-sphere is the unital $*$-subalgebra
\[
\C[\S^4_\te] = \Tex{span}_{\C} \{ f \ot U^k \mid k \in \zz^2 \, , \, \, f \in C^\infty(\S^4)_k \} \subset C^\infty( \S^4) \ot \C[ \T^2_\te].
\]
The unital $C^*$-algebra $C(\S^4_\theta)$ is defined as the completion of $\C[\S^4_\te]$ with respect to the supremum-norm
\[
\| \cd \|_{\S^4_\te} : C^\infty( \S^4) \ot \C[ \T^2_\te] \to [0,\infty) \q \| x \|_{\S^4_\te} = \sup_{p \in \S^4} \| x(p) \|_{C(\T^2_\te)}.
\]
The smooth functions on the noncommutative 4-sphere $C^\infty(\S^4_\te)$ is the unital Fr\'echet $*$-algebra obtained as the closure of the coordinate algebra $\C[\S^4_\te]$ inside the projective tensor product of Fr\'echet $*$-algebras $C^\infty(\S^4) \hot C^\infty(\T^2_\te)$.
\end{defn}

Concerning the noncommutative analogue of the principal stratum $\S^4_0$ for the $\T^2$-action on $\S^4$, we define the coordinate algebra
\[
\C_c[ (\S^4_0)_\te] = \Tex{span}_{\C} \{ f \ot U^k \mid k \in \zz^2 \, , \, \, f \in C_c^\infty(\S^4_0)_k \} \subset \C[\S^4_\te]
\]
as well as the non-unital $C^*$-subalgebra $C_0\big( (\S^4_0)_\te \big) \subset C( \S^4_\te)$ obtained by taking the $C^*$-norm-closure of the coordinate algebra $\C_c[ (\S^4_0)_\te]$. Moreover, we have the non-unital Fr\'echet $*$-algebra of smooth functions vanishing at infinity $C^\infty_0(( \S^4_0)_\te)$ obtained as the closure of $\C_c[ (\S^4_0)_\te]$ inside the Fr\'echet $*$-algebra $C^\infty(\S^4_\te)$.


The key observation needed for introducing a Dirac operator on the noncommutative 4-sphere is that the spin$^c$ structure on the commutative 4-sphere is $\T^2$-equivariant, see \cite{LS04,LS06}. Let us denote the Clifford module of smooth sections of the $\zz/2\zz$-graded spinor bundle over the 4-sphere by $\sE_{\S^4}$ and the associated $\zz/2\zz$-graded Hilbert space of $L^2$-spinors by $L^2( \sE_{\S^4})$. The $\T^2$-equivariance of the spin$^c$ structure entails the existence of a strongly continuous even unitary representation
\[
\wit \al : \T^2 \to \U( L^2(\sE_{\S^4}))
\]
which restricts to an even action $\wit \al : \T^2 \ti \sE_{\S^4} \to \sE_{\S^4}$ on the smooth sections of the spinor bundle (referred to as the spin$^c$ lift). This action lifts the action of $\T^2$ on $C^\infty(\S^4)$ in the sense that
\begin{equation}
\label{eq:equivariance}
\tilde \alpha_t( f \cdot \xi) = \alpha_t(f) \cdot \tilde \alpha_t (\xi) \q f \in C^\infty(\S^4) \, , \, \, \xi \in \sE_{\S^4} \, , \, \, t \in \T^2.
\end{equation}
Moreover, the Dirac operator $D_{\S^4}$ on the commutative 4-sphere is equivariant meaning that
\[
D_{\S^4}( \wit \al_t (\xi) ) = \wit \al_t(  D_{\S^4}(\xi) ) \q \Tex{for all } \xi \in \sE_{\S^4}\, , \, \, t \in \T^2.
\] 

In order to define a left action of the noncommutative 4-sphere $C(\S^4_\te)$ on the Hilbert space $L^2(\sE_{\S^4})$ we embed this Hilbert space into the tensor product of Hilbert spaces $L^2(\sE_{\S^4}) \hot L^2(\T^2_\te)$. To this end, for each $k \in \zz^2$, we define the spectral subspace
\[
L^2(\sE_{\S^4})_k = \{ \xi \in L^2(\sE_{\S^4}) \mid \wit \al_t(\xi) = t^k \cd \xi \Tex{ for all } t \in \T^2 \} \su L^2(\sE_{\S^4})
\]
and notice that $L^2(\sE_{\S^4})$ is unitarily isomorphic to the Hilbert space direct sum $\bop_{k \in \zz^2} L^2(\sE_{\S^4})_k$. We may thus define the isometry
\[
V : L^2(\sE_{\S^4}) \to L^2(\sE_{\S^4}) \hot L^2(\T^2_\te) \q V( \sum_{k \in \zz^2} \xi_k ) = \sum_{k \in \zz^2} \xi_k \ot U^k.
\]
The left action of $C(\S^4_\te)$ on the Hilbert space of $L^2$-spinors is then given by the $*$-homomorphism
\[
\pi : C(\S^4_\te) \to \sL( L^2(\sE_{\S^4})) \q \pi(x)(\xi) = (V^* x V)(\xi),
\]
where we are suppressing the tensor product representation of $C(\S^4_\te)$ as bounded operators on $L^2(\sE_{\S^4}) \hot L^2(\T_\te^2)$.

We quote the following result from \cite{CL01}, see also \cite{CD02}:

\begin{prop}[Connes-Landi]
The triple $( C^\infty(\S^4_\theta), L^2(\sE_{\S^4}), D_{\S^4})$ is an even $4^+$-summable spectral triple. 
\end{prop}
\begin{proof}
Since the unbounded operator $D_{\S^4}$ agrees with (the closure of) the classical Dirac operator on $\S^4$, it suffices to check that each $x \in C^\infty(\S^4_\theta)$ preserves the domain, $\Tex{Dom}(D_{\S^4})$, and has a bounded commutator with $D_{\S^4}$. To this end, we define the unbounded operator $D_{\S^4} \hot 1$ as the closure of the symmetric unbounded operator $D_{\S^4} \ot 1 : \Tex{Dom}(D_{\S^4}) \ot L^2(\T^2_\te) \to L^2(\sE_{\S^4}) \hot L^2(\T^2_\te)$. Using the tensor product representation of $C(\S^4_\te)$ on $L^2(\sE_{\S^4}) \hot L^2(\T_\te^2)$ we obtain that our $x \in C^\infty(\S^4_\te)$ preserves the domain of $D_{\S^4} \hot 1$ and has a bounded commutator with this unbounded operator. Moreover, using the equivariance of $D_{\S^4}$ with respect to the spin$^c$ lift $\wit \al$, one obtains that
\[
(D_{\S^4} \hot 1)V = V D_{\S^4} \q \Tex{and} \q D_{\S^4} V^* = V^* (D_{\S^4} \hot 1).
\]
Combining these observations, we see that $[ D_{\S^4}, \pi(x) ] = [D_{\S^4}, V^* x V] = V^* [D_{\S^4} \hot 1,x] V$ (as operators on $\Tex{Dom}(D_{\S^4})$), and the proposition is proved.
\end{proof}

%
Note that when we restrict the spinor bundle to the principal stratum $\S^4_0$ we end up with a trivial bundle $\S^4_0 \times \C^4 \to \S^4_0$. Under this identification, the spin$^c$ lift of the torus action turns out to be trivial in the fibers and only acting on the base $\S^4_0$. Thus $\tilde \alpha =  \alpha \otimes \Tex{id}$, when restricted to the smooth compactly supported sections, $C_c^\infty(\S^4_0) \otimes \C^4$, of this trivial bundle. 

Moreover, a ``local'' expression for the restriction of the Dirac operator $D_{\S^4}$ to $C_c^\infty(\S^4_0) \ot \C^4$ can be given:
\begin{align}
\label{eq:dirac-S4}
&D_{\S^4}|_{C_c^\infty(\S^4_0) \ot \C^4} = i \frac{1}{\cos \varphi \cos \psi} \gamma^1 \frac{\pa}{\pa \te_1} + i \frac{1}{\sin \varphi \cos \psi} \gamma^2 \frac{\pa}{\pa \te_2} \\ \nonumber
&\quad + i \frac{1}{\cos \psi} \gamma^3 \left( \frac{\partial}{\partial \varphi}+ \frac 12 \cot \varphi - \frac 12 \tan \varphi \right)
+ i \gamma^4 \left( \frac{\partial}{\partial \psi}- \frac 32 \tan \psi\right)
\end{align}
in terms of four flat Dirac gamma matrices $\gamma^j \in M_4(\C)$. Up to unitary equivalence, this is precisely the local expression that appeared in \cite{CL01}. 

Since the complement of the principal stratum $\S^4_0\subset \S^4$ can be written as the union 
\[
\S^4 \sem \S^4_0 = (\S^4)^{\T \ti \{0\}} \cup (\S^4)^{\{0\} \ti \T}
\]
of compact embedded submanifolds, each of codimension strictly greater than one, it follows from \cite[Proposition 30]{KS17b} that $C^\infty_c(\S^4_0) \otimes \C^4 \subset L^2(\sE_{\S^4})$ is a core for the Dirac operator $D_{\S^4}$.

\section{Factorization of the spectral triple on $\S^4_\te$}
Throughout this section $\te = \te_{12} \in \rr$ will be a fixed deformation parameter. 

As a preparation for the explicit tensor sum factorization of the Connes-Landi spectral triple \newline $(C^\infty(\S^4_\te), L^2(\sE_{\S^4}), D_{\S^4})$, we first introduce a $C^*$-correspondence $X$ from $C_0((\S^4_0)_\theta)$ to $C_0(Q^2_0)$. This $C^*$-correspondence will carry a vertical Dirac operator yielding an unbounded Kasparov module from $C_0( (\S^4_0)_\te)$ to $C_0(Q^2_0)$.

As a Hilbert $C^*$-module over $C_0(Q_0^2)$, $X$ is the completion of the right module 
\[
\cX = C_c^\infty( \S^4_0) \ot \C^2
\]
over $C_c^\infty(Q_0^2)$ with respect to the norm coming from the $C_0(Q^2_0)$-valued inner product:
\begin{equation}\label{eq:innex}
\begin{split}
& \inn{\xi , \eta}_X\big( \cos \varphi \cos \psi, \sin\varphi \cos\psi, \sin\psi\big)  \\
& \q = \int_{\T^2} (\xi^* \cd \eta)\big( t_1^{-1}\cd \cos\varphi \cos \psi, t_2^{-1} \cd \sin\varphi \cos\psi, \sin\psi \big) dt \\
& \qqq
\cd \sin \varphi \cos \varphi \cos^2 \psi \q \xi,\eta \in  \cX.
\end{split}
\end{equation}
Clearly, $X$ is $\zz/2\zz$-graded with grading operator 
\[
\ga_X = \ma{cc}{1 & 0 \\ 0 & -1}.
\]
The action $\wit \al = \al \ot \Tex{id} : \T^2 \ti \cX \to \cX$ induces an even strongly continuous action $\wit \al : \T^2 \ti X \to X$ such that $\wit \al_t : X \to X$ is a unitary operator for all $t \in \T^2$. 

In order to introduce the left action of $C( (\S_0^4)_\te)$ on $X$, we define the spectral subspace
\[
X_k = \big\{ \xi \in X \mid \wit \al_t(\xi) = t^k \cd \xi \big\}
\]
for each $k \in \zz^2$. It then holds that $X$ is unitarily isomorphic to the Hilbert $C^*$-module direct sum $\bop_{k \in \zz^2} X_k$, in particular we have a bounded adjointable isometry
\[
V : X \to X \hot L^2(\T_\te^2) \q V( \sum_{k \in \zz^2} \xi_k) = \sum_{k \in \zz^2} \xi_k \ot U^k.
\]
Notice that the Hilbert $C^*$-module $X \hot L^2(\T_\te^2)$ over $C_0(Q_0^2)$ carries a left action of the non-unital $C^*$-algebra $C_0( (\S^4_0)_\te)$ coming from the tensor product of the representations of $C_0(\S^4_0)$ of $C(\T^2_\te)$ as multiplication operators on $X$ and $L^2(\T^2_\te)$, respectively. The left action of $C_0(( \S^4_0)_\te)$ on the Hilbert $C^*$-module $X$ is then determined by the $*$-homomorphism
\begin{equation}\label{eq:leftvert}
\pi : C_0(( \S^4_0)_\te) \to \sL(X) \q \pi(x) = V^* x V.
\end{equation}

We now introduce a vertical Dirac operator $D_V : \Tex{Dom}(D_V) \to X$ as follows. Let us choose the two Pauli matrices 
\begin{equation}\label{eq:pauli}
\sigma^1 = \ma{cc}{0 & 1 \\ 1 & 0} \q \Tex{and} \q \sigma^2 = \ma{cc}{0 & -i \\ i & 0}
\end{equation}
and define the odd unbounded operator $\sD_V : \cX \to X$ by
\begin{equation}\label{eq:vertunbo}
\sD_V = i \frac{1}{\cos \varphi \cos \psi}\sigma^1 \frac{\pa}{\pa \te_1} +
 i \frac{1}{\sin \varphi \cos \psi} \sigma^2 \frac{\pa}{\pa \te_2}.
\end{equation}
This odd unbounded operator is symmetric with respect to the inner product $\langle \cdot,\cdot \rangle_X$ and we denote its closure by $D_V: \dom(D_V) \to X$. Remark that $D_V$ is equivariant in the sense that
\[
D_V \wit{\al}_t = \wit{\al}_t D_V \q \Tex{for all } t \in \T^2.
\]

The vertical part of our geometry can now be summarized in the following:

\begin{prop}\label{p:vertsphe}
The triple $\big(C^\infty_0( (\S^4_0)_\te), X,D_V\big)$ defines an even unbounded Kasparov module from $C_0((\S^4_0)_\te)$ to $C_0(Q^2_0)$ with grading operator $\ga_X$.
\end{prop}
\proof
Let us start by proving that each $x \in C^\infty_0( (\S^4_0)_\te)$ preserves the domain of $D_V$ and admits a bounded commutator with $D_V$. We define $D_V \hot 1 : \Tex{Dom}(D_V \hot 1) \to X \hot L^2(\T^2_\te)$ as the closure of the symmetric unbounded operator $D_V \ot 1 : \Tex{Dom}(D_V) \ot L^2(\T^2_\te) \to X \hot L^2(\T^2_\te)$. Since $D_V$ is equivariant with respect to the torus action on $X$ we obtain the identities
\[
(D_V \hot 1) V = V D_V \q V^* (D_V \hot 1) = D_V V^*
\]
on the domain of $D_V$ and the domain of $D_V \hot 1$, respectively. Now, since $D_V$ is the closure of a first order differential operator, the element $x \in C^\infty_0( (\S^4_0)_\te)$, considered as a bounded adjointable operator on $X \hot L^2(\T^2_\te)$, preserves the domain of $D_V \hot 1$ and has a bounded commutator with $x$. Our commutator statement therefore follows from the identity
\[
[ D_V, \pi( x) ] = [D_V, V^* x  V] = V^* [D_V \hot 1,x] V,
\]
which holds on the domain of $D_V$.

The fact that $D_V$ is regular and selfadjoint and that it has a locally compact resolvent follows since $D_V$ is the closure of a symmetric, vertically elliptic, first order differential operator, \cite[Theorem 3]{KS17b}. However, in the case at hand we may also prove this directly as follows:

Let us define vectors in $C^\infty(\S^4_0) \ot \C^2$ by
\[
\Psi_{n_1,n_2}^\pm = \begin{pmatrix} e^{i \te_1 n_1} e^{i \te_2 n_2} \\ \pm c(n_1,n_2) e^{i \te_1 n_1} e^{i \te_2 n_2} \end{pmatrix} \cd \frac{1}{\sqrt{2 \sin \varphi \cos \varphi \cos^2 \psi}},
\]
where
\[
c(n_1,n_2) =\fork{ccc}{ \frac{\left(\frac{n_1}{\cos \varphi} + i \frac{n_2}{\sin \varphi} \right)}{\sqrt{\frac{n_1^2}{\cos^2\varphi} + \frac{n_2^2}{\sin^2\varphi}}} & \Tex{for} & (n_1,n_2) \neq (0,0) \\ 
1 & \Tex{for} & (n_1,n_2) = (0,0)}.
\]
One then checks that 
\[
D_V \left(\Psi_{n_1,n_2}^\pm \cd f \right)  = \mp \lambda_{n_1,n_2} \cdot \Psi_{n_1,n_2}^\pm \cd f \q \Tex{for all } f \in C_c^\infty(Q_0^2),
\]
where
\[
\lambda_{n_1,n_2} = \sqrt{\frac{n_1^2}{\cos^2\varphi \cos^2\psi} + \frac{n_2^2}{\sin^2\varphi \cos^2\psi}}.
\]
Thus $\{\Psi_{n_1,n_2}^\pm\}$ is a family of (generalized) eigenvectors for $D_V$, varying over the base space $Q_0^2$ (see Figure \ref{fig:eigenval-S} for a plot of the first eigenvalues). Moreover, we remark that the $\C$-linear span
\[
\Tex{span}\big\{ \Psi_{n_1,n_2}^{\pm} \cd f \mid f \in C_c^\infty(Q_0^2) \big\}
\]
is norm-dense in $X$ and that our (generalized) eigenvectors form an orthonormal system in the sense that
\[
\begin{split}
& \inn{ \Psi^{\pm}_{n_1,n_2} \cd f, \Psi^{\pm}_{m_1,m_2} \cd g}_X = \fork{ccc}{f^* g & \Tex{for} & (n_1,n_2) = (m_1,m_2) \\ 0 & \Tex{for} & (n_1,n_2) \neq (m_1,m_2)} \q \Tex{and} \\
& \inn{ \Psi^{\pm}_{n_1,n_2} \cd f, \Psi^{\mp}_{m_1,m_2} \cd g}_X = 0 \, \, , \, \, \, n_1,n_2,m_1,m_2 \in \zz ,
\end{split}
\]
for all $f,g \in C_c^\infty(Q_0^2)$.

For each $n_1,n_2 \in \zz$ and each $\mu \in \rr\sem\{0\}$ we define the bounded adjointable operators
\[ 
\begin{split}
& K_{n_1,n_2}^+(\mu) = (i \mu  - \la_{n_1,n_2} )^{-1}\ket{\Psi_{n_1,n_2}^+} \bra{\Psi_{n_1,n_2}^+}\q \Tex{and} \\
& K_{n_1,n_2}^-(\mu) = (i \mu + \la_{n_1,n_2} )^{-1}\ket{\Psi_{n_1,n_2}^-} \bra{\Psi_{n_1,n_2}^-}
\end{split}
\]
on $X$. We emphasize that $K_{n_1,n_2}^{\pm}(\mu)$ is not a compact operator on $X$ but that $\pi(x) \cd K_{n_1,n_2}^{\pm}(\mu)$ is a compact operator on $X$ for all $x \in C_0\big( (\S_0^4)_\te\big)$.

It may then be verified that the resolvent $(i \mu + D_V)^{-1} : X \to X$ is given by the bounded adjointable operator
\[
\sum_{n_1,n_2 \in \zz} \big(K_{n_1,n_2}^+(\mu) + K_{n_1,n_2}^-(\mu) \big),
\]
where the sum converges in the operator norm on $X$. This proves that the unbounded operator $D_V$ is selfadjoint and regular since these properties are equivalent to the existence of the above resolvents, see \cite[Chapter 9]{Lan95}. Since we moreover have that 
\[
\sum_{n_1,n_2 \in \zz} \pi(x) \cd \big(K_{n_1,n_2}^+(\mu) + K_{n_1,n_2}^-(\mu) \big)
\]
is a compact operator for all $x \in C_0\big( (\S_0^4)_\te\big)$, this ends the proof of the proposition. \endproof

\begin{figure}
\includegraphics[scale=.4]{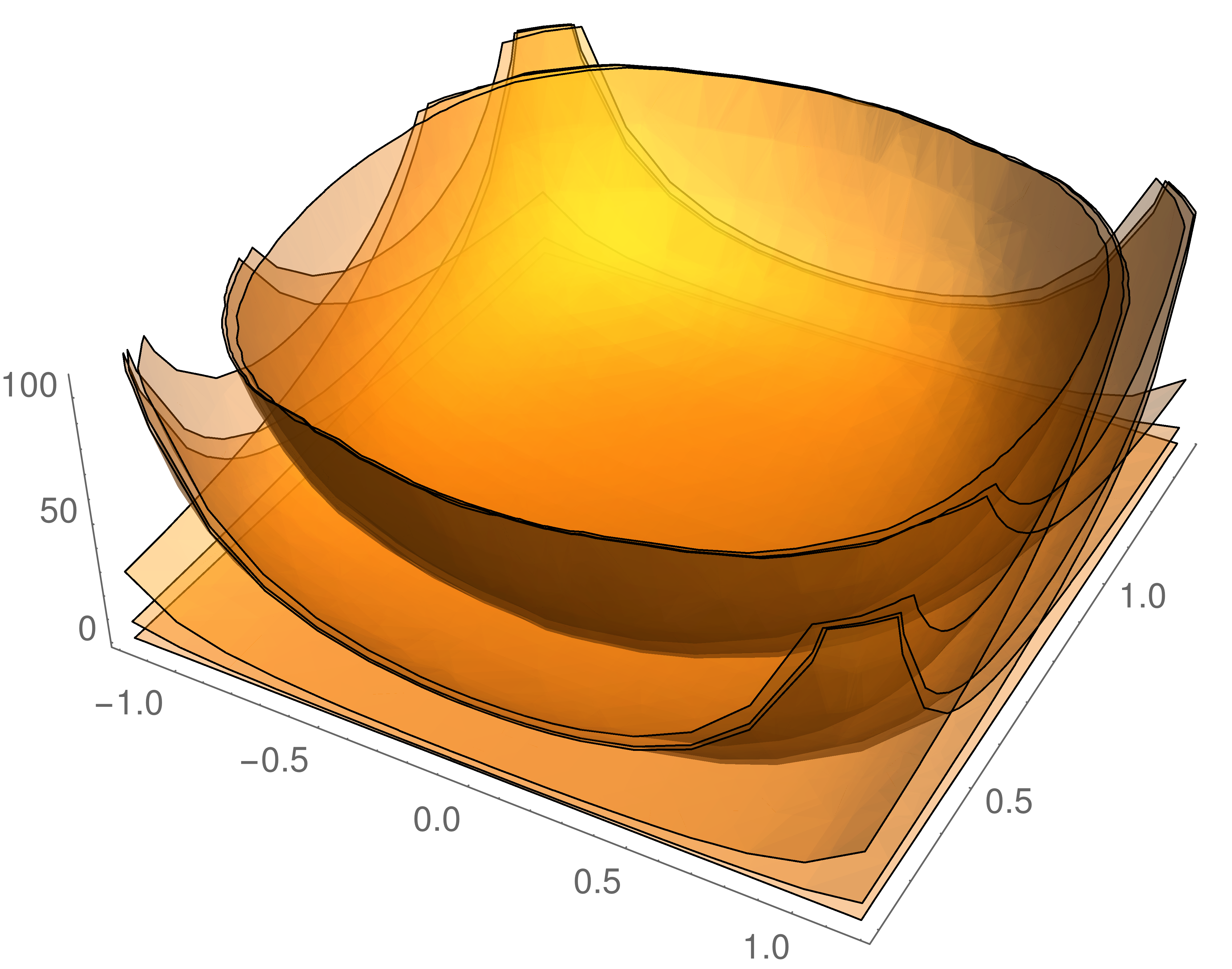}
\caption{The first few families of eigenvalues of $D_V$, varying over the quadrant $Q_0^2$.}
\label{fig:eigenval-S}
\end{figure}

We now turn to the horizontal part of our geometry.

We define the unbounded operator $\sD_{Q_0^2} : C_c^\infty (Q^2_0) \otimes \C^2 \to L^2(Q^2_0) \otimes \C^2$ as the restriction of the Dirac operator on $\S^2$ to the open quadrant $Q^2_0$, thus
\[
\sD_{Q_0^2}  = i \frac{1}{\cos\psi} \sigma^1 \frac{\pa}{\pa \varphi} + i \sigma^2 \left( \frac{\pa}{\pa \psi} - \frac12 \tan \psi \right),
\]
where $\si^1$ and $\si^2 \in M_2(\C)$ are the Pauli matrices from Equation \eqref{eq:pauli}. The unbounded operator $\sD_{Q_0^2}$ is symmetric and it is odd with respect to the $\zz/2\zz$-grading of the Hilbert space $L^2(Q_0^2) \ot \C^2$ given by the grading operator 
\[
\ga_{Q_0^2} = \ma{cc}{1 & 0 \\ 0 & -1}.
\]
We denote the closure of $\sD_{Q_0^2}$ by $D_{Q_0^2}: \dom(D_{Q_0^2}) \to  L^2(Q^2_0) \otimes \C^2$. We recall that the inner product on $L^2(Q_0^2)$ comes from the Riemannian metric $g_{Q_0^2} = \cos^2 \psi ~ d \varphi^2 + d \psi^2$ and that $C_0(Q_0^2)$ acts from the left on $L^2(Q_0^2)$ as pointwise multiplication operators.

The properties of our horizontal data can then be summarized in the following proposition, see \cite{BDT89,Hil10}:

\begin{prop}
The triple $(C_c^\infty(Q_0^2), L^2(Q_0^2) \otimes \C^2, D_{Q_0^2})$ defines an even half-closed chain from $C_0(Q^2_0)$ to $\C$ with grading operator $\ga_{Q_0^2}$. 
\end{prop}

We record that the right module structure on $\cX$ induces an equivariant even isomorphism
\begin{equation}
\label{eq:map-u-S4}
W : \cX \otimes_{C_c^\infty(Q^2_0)} ( C_c^\infty(Q^2_0) \otimes \C^2)  \to C_c^\infty(\S^4_0) \otimes \C^4 ,
\end{equation}
which extends to a unitary isomorphism 
\[
W : X \hot_{C_0(Q^2_0)} ( L^2(Q_0^2) \otimes \C^2 ) \to L^2(\S^4_0)\otimes \C^4
\]
of $\zz/2\zz$-graded $C^*$-correspondences from $C_0( (\S^4_0)_\te)$ to $\C$. We are here also identifying $\C^2 \ot \C^2$ with $\C^4$ via the map 
\[
\la \ot \ma{c}{\mu_1 \\ \mu_2} \mapsto \ma{c}{\la \mu_1 \\ \la \mu_2}.
\]
Remark that the grading on $X \hot_{C_0(Q_0^2)} (L^2(Q_0^2) \ot \C^2)$ is given by the tensor product 
\[
\ga = \ga_X \hot \ga_{Q_0^2}
\]
of grading operators and hence that the grading operator on $L^2(\S^4_0) \ot \C^4$ is given by
\[
\ga = \ma{cccc}{1 & 0 & 0 & 0 \\ 0 & -1 & 0 & 0 \\ 0 & 0 & -1 & 0 \\ 0 & 0 & 0 & 1}.
\]
 In particular, we have the even selfadjoint unitary operator
\[
\begin{split}
\Ga = (\ga_X \hot 1) \frac{1 + \ga}{2} + \frac{1-\ga}{2} & : X \hot_{C_0(Q^2_0)} ( L^2(Q_0^2) \otimes \C^2 ) \\
& \q \to X \hot_{C_0(Q^2_0)} ( L^2(Q_0^2) \otimes \C^2 ).
\end{split}
\]
\medskip

In order to introduce the tensor sum $D_V \times_\nabla D_{Q_0^2}$ we need to lift the operators $D_V$ and $D_{Q_0^2}$ to the interior tensor product\newline $X \hot_{C_0(Q^2_0)} ( L^2(Q^2_0) \otimes \C^2)$. 

For $D_V$ this is straightforward, we simply consider the symmetric unbounded operator 
\[
\begin{split}
& \sD_V \otimes 1 : \cX \otimes_{C_c^\infty(Q^2_0)} ( C_c^\infty(Q^2_0) \otimes \C^2) \to X \hot_{C_0(Q_0^2)} ( L^2(Q^2_0) \otimes \C^2 ) \\
& (\sD_V \ot 1) (\xi \ot \eta) = \sD_V(\xi) \ot \eta.
\end{split}
\]
After identifying the domain of $\sD_V \ot 1$ with $C_c^\infty(\S^4_0) \otimes \C^4$ using the isomorphism $W$ defined in Equation \eqref{eq:map-u-S4}, we have that
\begin{equation}
\label{eq:tensor-sum-S-S4}
W \Ga (\sD_V \ot 1) \Ga W^*(s) = i \frac{1}{\cos \varphi \cos \psi} \gamma^1 \frac{\pa}{\pa \te_1} (s) + i \frac{1}{\sin \varphi \cos \psi} \gamma^2 \frac{\pa}{\pa \te_2}(s),
\end{equation}
for all $s \in C_c^\infty(\S^4_0) \otimes \C^4$, upon writing
\[
\gamma^1 = \ma{cccc}{0 & 1 & 0 & 0 \\ 1 & 0 & 0 & 0 \\ 0 & 0 & 0 & -1 \\ 0 & 0 & -1 & 0} \q \Tex{and} \q 
\gamma^2= \ma{cccc}{0 & - i & 0 & 0 \\ i & 0 & 0 & 0 \\ 0 & 0 & 0 & i \\ 0 & 0 & -i & 0}.
\]

To lift the unbounded operator $D_{Q_0^2}$ to the interior tensor product, we need an even and equivariant metric connection on the Hilbert $C^*$-module $X$. In view of the expression \eqref{eq:mean-curv-S4} for the mean curvature, as in \cite{KS17b}, we introduce the even equivariant connection:
\begin{align*}
& \nabla_{\pa/\pa\varphi} = \frac{\pa}{\pa \varphi} + \frac 12 k\left( \frac{\pa}{\pa\varphi}\right) =  \frac{\pa}{\pa \varphi} + \frac 12 \cot \varphi - \frac 12 \tan \varphi\\
&\nabla_{\pa/\pa\psi} = \frac{\pa}{\pa \psi} + \frac 12 k\left( \frac{\pa}{\pa\psi}\right) =  \frac{\pa}{\pa \psi} - \tan \psi\\
\end{align*}
as $\C$-linear maps from $\cX \to X$. It can be checked directly, using the form of the inner product on $X$ from Equation \eqref{eq:innex}, that this connection is also metric, thus that
\[
\begin{split}
\pa/\pa \varphi\big( \inn{\xi,\eta}_X \big) & = 
\binn{\Na_{\pa/\pa \varphi} (\xi), \eta}_X + \binn{\xi, \Na_{\pa/\pa \varphi} (\eta)}_X \q \Tex{and} \\
\pa/\pa \psi\big( \inn{\xi,\eta}_X \big) & = 
\binn{\Na_{\pa/\pa \psi} (\xi), \eta}_X + \binn{\xi, \Na_{\pa/\pa \psi} (\eta)}_X
\end{split}
\]
for all $\xi,\eta \in \cX$. The resulting symmetric unbounded operator $1 \ot_\nabla \sD_{Q_0^2}$ is given by
\begin{equation}\label{eq:tensor-sum-T-S4}
\begin{split}
W(1 \ot_\nabla \sD_{Q_0^2})W^*(s) & =  i \frac{1}{\cos \psi} \gamma^3 \left( \frac{\partial}{\partial \varphi} + \frac 12 \cot \varphi - \frac 12 \tan \varphi \right)s \\
& \q + i \gamma^4 \left( \frac{\partial}{\partial \psi}- \frac 32 \tan \psi\right)s,
\end{split}
\end{equation}
for all $s \in C_c^\infty(\S^4_0) \otimes \C^4$, where we have written 
\[
\gamma^3 = \ma{cccc}{0 & 0 & 1 & 0 \\ 0 & 0 & 0 & 1 \\ 1 & 0 & 0 & 0 \\ 0 & 1 & 0 & 0} \q \Tex{and} \q
\ga^4 = \ma{cccc}{0 & 0 & -i & 0 \\ 0 & 0 & 0 & -i \\ i & 0 & 0 & 0 \\ 0 & i & 0 & 0}.
\]
The tensor sum we are after is given by 
\begin{align*}
&\sD_V \times_\nabla \sD_{Q_0^2} = \sD_V \ot 1 + (\ga_X \ot 1)(1 \ot_\nabla \sD_{Q_0^2})
\\
& \qquad : \cX \otimes_{C_c^\infty(Q^2_0)} ( C_c^\infty(Q^2_0 )\otimes \C^2) \to X \otimes_{C_0(Q^2_0)} ( L^2(Q^2_0)\otimes \C^2).
\end{align*}
Clearly $\sD_V \times_\nabla \sD_{Q_0^2}$ is a symmetric operator and we denote its closure by 
\[
D_V \times_\nabla D_{Q_0^2}: \dom(D_V \times_\nabla D_{Q_0^2}) \to X \otimes_{C_0(Q^2_0)} ( L^2(Q^2_0)\otimes \C^2).
\]

\begin{thm}
\label{thm:tensor-sum}
We have the equality of selfadjoint unbounded operators
\[
W \Ga (D_V \times_\nabla D_{Q_0^2}) \Ga W^* = D_{\S^4}.
\]
\end{thm}
\proof
Since $\cX \otimes_{C_c^\infty(Q^2_0)} C_c^\infty(Q^2_0) \otimes \C^2$ is a core for $D_V \times_\nabla D_{Q_0^2}$ and since this core is isomorphic to the core $C_c^\infty(\S^4_0) \otimes \C^4$ of $D_{\S^4}$ via the unitary operator 
\[
W \Ga : X \hot_{C_0(Q_0^2)} ( L^2(Q_0^2) \ot \C^2) \to L^2( \S_0^4) \ot \C^4,
\]
it suffices to check the desired identity on the core $C_c^\infty(\S^4_0) \ot \C^4$. But here it follows readily from Equations \eqref{eq:tensor-sum-S-S4}, \eqref{eq:tensor-sum-T-S4} and the form of the Dirac operator in \eqref{eq:dirac-S4} that
\[
(W \Ga (D_V \times_\nabla D_{Q_0^2}) \Ga W^*)(s) = D_{\S^4}(s) \q s \in C_c^\infty(\S^4_0) \ot \C^4.
\]
Notice to this end that $\Ga (\ga_X \ot 1)(1 \ot_\nabla \sD_{Q_0^2}) \Ga = 1 \ot_\Na \sD_{Q_0^2}$.
\endproof

We will now put the above factorization result in the context of spectral triples and KK-theory. First, consider the map $C^\infty_c(\S^4_0) \to C^\infty(\S^4)$ given by extension by zero. Since this map is $\T^2$-equivariant and continuous, it induces a  continuous $*$-homomorphism $\io : C^\infty_0((\S^4_0)_\theta) \to C^\infty(\S^4_\theta)$ (and a $*$-homomorphism $\io : C_0((\S^4_0)_\theta) \to C(\S^4_\theta)$ at the level of $C^*$-algebras). We may use it to pullback the even spectral triple $(C^\infty(\S^4_\theta), L^2(\sE_{\S^4}), D_{\S^4})$ to an even spectral triple 
\[
\io^*( C^\infty(\S^4_\te), L^2(\sE_{\S^4}), D_{\S^4} ) = ( C^\infty_0((\S^4_0)_\theta), L^2(\sE_{\S^4}), D_{\S^4}) .
\]
At the level of bounded $KK$-theory this pullback operation corresponds to the usual pullback homomorphism
\[
\io^* : KK_0( C(\S^4_\theta), \C) \to KK_0( C_0((\S^4_0)_\theta), \C).
\]
From the above theorem we thus obtain (up to unitary equivalence) that
\begin{equation}\label{eq:restric}
\begin{split}
& \io^*( C^\infty(\S^4_\theta), L^2(\sE_{\S^4}), D_{\S^4} ) \\
& \q = ( C^\infty_0((\S^4_0)_\theta), X \hot_{C_0(Q_0^2)} (L^2(Q_0^2) \ot \C^2), D_V \ti_{\Na} D_{Q_0^2}).
\end{split}
\end{equation}
Moreover, our tensor sum decomposition of $D_{\S^4}$ corresponds to the interior Kasparov product in bivariant K-theory, as we now show:

Let us denote the classes in even $KK$-theory associated to the unbounded Kasparov module \newline $(C^\infty_0((\S^4_0)_\te), X,D_V)$, the half-closed chain $(C^\infty_c(Q_0^2), L^2(Q_0^2) \ot \C^2, D_{Q_0^2})$ and the spectral triple $(C^\infty(\S^4_\te) , L^2(\sE_{\S^4}), D_{\S^4})$ by 
\[
\begin{split}
& [(D_V)_\te] \in KK_0(C_0( (\S^4_0)_\te), C_0(Q_0^2)) \\
& [D_{Q_0^2}] \in KK_0( C_0(Q_0^2),\C) \q \Tex{and} \q [D_{\S^4_\te}] \in KK_0( C(\S^4_\te),\C),
\end{split}
\]
respectively.

\begin{thm}
\label{thm:fact}
The even spectral triple $( C^\infty_0((\S^4_0)_\theta), L^2(\sE_{\S^4}), D_{\S^4})$ is the unbounded Kasparov product of the even unbounded Kasparov module $(C^\infty_0((\S^4_0)_\theta),X,D_V)$ and the even half-closed chain $(C_c^\infty(Q_0^2), L^2(Q^2_0) \otimes \C^2, D_{Q_0^2})$. In particular it holds that
\[
\io^*[ D_{\S^4_\te}] = [ (D_V)_\te] \hot_{C_0(Q_0^2)} [D_{Q_0^2}]
\]
at the level of bounded $KK$-theory. 
\end{thm}
\proof
We will show that the class $[(D_V\times_\nabla D_{Q_0^2})_\te]$ in $KK_0( C_0((\S^4_0)_\te),\C)$ represented by the even spectral triple $( C^\infty_0((\S^4_0)_\theta), X \hot_{C_0(Q_0^2)} (L^2(Q_0^2) \ot \C^2), D_V \ti_{\Na} D_{Q_0^2})$ is the KK-product of the classes $[(D_V)_\te]$ and $[D_{Q_0^2}]$ using a generalization to half-closed chains of a theorem by Kucerovsky \cite[Theorem 13]{Kuc97}. This generalization was proved recently by the authors, see \cite{KS17a}. 

Thus, we need to check the connection condition, \cite[Definition 27]{KS17a}, and the local positivity condition, \cite[Definition 29]{KS17a}, in the case at hand.

The connection condition is satisfied by a standard computation so we focus on the more subtle local positivity condition. Indeed, the connection condition follows since, for each $\xi \in \cX$, we have that
\[
\begin{split}
& (D_V \ti_{\Na} D_{Q_0^2})(\xi \ot \eta) - \ga_X(\xi) \ot D_{Q_0^2}(\eta) \\
& \q = D_V(\xi) \ot \eta + \ga_X  \Na_{\frac{\pa}{\pa \varphi}}(\xi) \cd \frac{1}{\cos \psi} \ot i \si^1 \cd \eta
+ \ga_X \Na_{\frac{\pa}{\pa \psi}}(\xi) \ot i \si^2 \cd \eta,
\end{split}
\]
for all $\eta \in C_c^\infty(Q_0^2) \ot \C^2$.

To verify the local positivity condition, we choose a countable approximate identity $\{ f_n \}$ for the $C^*$-algebra $C_0(Q_0^2)$ such that each $f_n$ is a smooth compactly supported function on $Q_0^2$. Letting $q : \S_0^4 \to Q_0^2$ denote the quotient map we then have that
\[
\La = \{ (f_n \ci q) \ot 1 \in \C_c[ (\S_0^4)_\te] \mid n \in \nn \} \su C^\infty_0( (\S_0^4)_\te)
\]
is a localizing subset in the sense of \cite[Definition 28]{KS17a}. Indeed, one verifies immediately that the commutator
\[
\begin{split}
[ \sD_V \ot 1 , \pi(x) \ot 1 ] & : \cX \ot_{C_c^\infty(Q_0^2)} ( C_c^\infty(Q_0^2) \ot \C^2) \\
& \q \to X \hot_{C_0(Q_0^2)} (L^2(Q_0^2) \ot \C^2)
\end{split}
\]
is trivial for all $x \in \La$, where the representation $\pi : C_0( (\S_0^4)_\te) \to \sL(X)$ is defined in \eqref{eq:leftvert}. Moreover, letting $D_V \hot 1$ denote the closure of $\sD_V \ot 1$, the inclusion
\[
\dom( D_V \ti_{\Na} D_{Q_0^2}) \cap \Tex{Im}( \pi(x) \hot 1) \su \dom(D_V \hot 1)
\]
can be proved exactly as in \cite[Lemma 21]{KS17b} (using G\aa rding's inequality and the ellipticity of $\sD_V \ti_{\Na} \sD_{Q_0^2}$). In fact, for each $x \in \La$, there exists a constant $C_x > 0$ such that
\begin{equation}\label{eq:gaar}
\| (D_V \hot 1) \ze \| \leq C_x \big( \| \ze \| +  \| (D_V \ti_{\Na} D_{Q_0^2}) \ze \| \big),
\end{equation}
for all $\ze \in \Tex{Im}( \pi(x) \hot 1) \cap \dom( D_V \ti_{\Na} D_{Q_0^2})$. 

To end the proof of the theorem, we only need to show that, for each $x \in \La$, there exists a constant $\ka_x > 0$ such that
\begin{equation}\label{eq:locaposi}
\begin{split}
& \binn{ (D_V \hot 1) (\pi(x) \hot 1)\xi, (D_V \ti_\Na D_{Q_0^2}) (\pi(x) \hot 1)\xi} \\ 
& \q + \binn{ (D_V \ti_\Na D_{Q_0^2}) (\pi(x) \hot 1)\xi, (D_V \hot 1) (\pi(x) \hot 1)\xi} \\
& \qq \geq - \ka_x\inn{(\pi(x) \hot 1)\xi,(\pi(x) \hot 1)\xi},
\end{split}
\end{equation}
for all $\xi \in \cX \ot_{C_c^\infty(Q_0^2)} ( C_c^\infty(Q_0^2) \ot \C^2)$.
%

Let now $x = (f_n \ci q) \ot 1 \in \La$ be fixed and choose the constant $\ka_x > 0$ such that
\[
\frac{1}{4} \Big( \frac{\tan^2 \varphi}{\cos^2 \psi} + 2 \tan^2(\psi) + \frac{\cot^2 \varphi}{\cos^2 \psi} \Big) \leq \ka_x
\]
for all $(\cos \varphi \cos\psi, \sin \varphi \cos \psi,\sin \psi)$ in the support of $f_n : Q_0^2 \to [0,1]$. Recall here that the support of $f_n$ is compact.

We define the symmetric unbounded differential operators
\[
\begin{split}
\sD_V' & = i \frac{1}{ \cos \psi \cos \varphi} \si^1 \frac{\pa}{\pa \te_1} :  \cX \to X \\
\sD_V'' & = i \frac{1}{ \cos \psi \sin \varphi} \si^2 \frac{\pa}{\pa \te_2} :  \cX \to X
\end{split}
\]
and remark that $\sD_V = \sD_V' + \sD_V''$, see \eqref{eq:vertunbo}. Moreover, we define the symmetric unbounded multiplication operators
\[
\begin{split}
\sT' & = \frac{\tan \varphi}{\cos \psi} \cd \si^1 + \tan \psi \cd \si^2 : C_c^\infty(Q_0^2) \ot \C^2 \to L^2(Q_0^2) \ot \C^2 \\
\sT'' & = \frac{\cot \varphi}{\cos \psi} \cd \si^1 - \tan \psi \cd \si^2 : C_c^\infty(Q_0^2) \ot \C^2 \to L^2(Q_0^2) \ot \C^2 .
\end{split}
\]

On the core $\cX \ot_{C_c^\infty(Q_0^2)} ( C_c^\infty(Q_0^2) \ot \C^2)$ for $D_V \ti_{\Na} D_{Q_0^2}$, we compute the anti-commutator:
\[
\begin{split}
& (\sD_V \ot 1) (\ga_X \ot_{\Na} \sD_{Q_0^2})
+ (\ga_X \ot_{\Na} \sD_{Q_0^2}) (\sD_V \ot 1) \\
& \q = (\ga_X \ot 1)\big( [ \Na_{\frac{\pa}{\pa \varphi}} , \sD_V ]  \ot i \si^1 \frac{1}{\cos \psi} 
+ [\Na_{\frac{\pa}{\pa \psi}}, \sD_V] \ot i \si^2\big) \\
& \q = \ga_X \sD_V' \ot i \frac{\tan \varphi}{\cos \psi} \cd \si^1
- \ga_X \sD_V'' \ot i \frac{\cot \varphi}{\cos \psi} \cd \si^1 \\ 
& \qq + \ga_X \sD_V \ot i \tan \psi \cd \si^2 \\
& \q = i \ga_X \sD_V' \ot \sT' - i \ga_X \sD_V'' \ot \sT''.
\end{split}
\]
Still computing on our core for $D_V \ti_{\Na} D_{Q_0^2}$, we obtain the identity
\[
\begin{split}
& \big( \sD_V' \ot 1 + i \frac{1}{2} \ga_X \ot \sT' \big)
\big( \sD_V' \ot 1 - i \frac{1}{2} \ga_X \ot \sT' \big) \\
& \qq + \big( \sD_V'' \ot 1 - i \frac{1}{2} \ga_X \ot \sT'' \big)
\big( \sD_V'' \ot 1 + i \frac{1}{2} \ga_X \ot \sT'' \big) \\
& \q = \sD_V^2 \ot 1 + \frac{1}{4} \ot \big( (\sT')^2 + (\sT'')^2 \big) \\
& \qq + (\sD_V \ot 1) (\ga_X \ot_{\Na} \sD_{Q_0^2})
+ (\ga_X \ot_{\Na} \sD_{Q_0^2}) (\sD_V \ot 1) .
\end{split}
\]

Recall now that $x = (f_n \ci q) \ot 1$ and remark that
\[
\frac{1}{4} \cd \big( (\sT')^2 + (\sT'')^2 \big) 
= \frac{1}{4} \cd \big( \frac{\tan^2 \varphi}{\cos^2 \psi} + 2 \tan^2(\psi) + \frac{\cot^2 \varphi}{\cos^2 \psi} \big).
\]
Hence, by our choice of $\ka_x$, we have the inequality
\[
\frac{1}{4} \cd \big( (\sT')^2 + (\sT'')^2 \big) \leq \ka_x
\]
on the support of $f_n : Q_0^2 \to [0,1]$. 

Let now $\xi \in \cX \ot_{C_c^\infty(Q_0^2)}( C_c^\infty(Q_0^2) \ot \C^2)$ be given and put
\[
\ze := (\pi(x) \hot 1)(\xi).
\]
We thus infer from the above computations that
\[
\begin{split}
& \binn{ (\sD_V \hot 1) \ze, (\sD_V \ti_\Na \sD_{Q_0^2}) \ze} + \binn{ (\sD_V \ti_\Na \sD_{Q_0^2}) \ze, (\sD_V \hot 1) \ze} \\
& \q = \binn{ (\sD_V \ot 1) \ze, (\sD_V \ot 1) \ze} \\ 
& \qq + \binn{\big( \sD_V' \ot 1 - i \frac{1}{2} \ga_X \ot \sT' \big) \ze, \big( \sD_V' \ot 1 - i \frac{1}{2} \ga_X \ot \sT' \big) \ze } \\
& \qq + \binn{ \big( \sD_V'' \ot 1 + i \frac{1}{2} \ga_X \ot \sT'' \big) \ze , 
\big( \sD_V'' \ot 1 + i \frac{1}{2} \ga_X \ot \sT'' \big) \ze } \\
& \qq - \binn{ \ze, \big(1 \ot \frac{1}{4}( (\sT')^2 + (\sT'')^2) \big)\ze } \\
& \q \geq - \ka_x \inn{\ze,\ze},
\end{split}
\]
where the last inequality follows since $\ze \in \Tex{Im}(\pi(x) \hot 1)$. This proves the local positivity condition and hence the result of the theorem. \endproof

\begin{rem}
We remark that the 4-sphere gives an example of an almost-regular fibration since $\pi : \S^4_0 \to Q_0^2$ is a proper Riemannian submersion of spin$^c$ manifolds. The above Theorem \ref{thm:fact} would therefore also follow from the general techniques on factorization of Dirac operators on almost-regular fibrations developed in \cite{KS17b}. Indeed, we shall see in the next section how these techniques can be transferred to the case of theta-deformations of general toric spin$^c$ manifolds. The reader can in this respect compare with \cite[Proposition 29]{FR15b}. 

The case of the 4-sphere is however much more explicit and arguably one of the most interesting, we therefore found it worthwhile to spell out the details.  
%
\end{rem}

\section{Dirac operators on toric noncommutative manifolds}
We end this paper by showing how one can generalize the above results on the 4-sphere to factorize Dirac operators on any toric noncommutative spin$^c$ manifold subject to a condition on the principal stratum. This will be based on our general factorization results of Dirac operators on almost-regular fibrations \cite{KS16,KS17b}. 

We consider a compact Riemannian spin$^c$ manifold $M$ with an isometric action of the n-torus $\T^n$ and denote the corresponding strongly continuous action on the $C^*$-algebra of continuous functions by $\alpha: \T^n \ti C(M) \to C(M)$.

Let us denote the spinor bundle on $M$ by $E_M \to M$, the Clifford bundle by $\Cl(T M) \to M$ and a fixed hermitian Clifford connection by
\[
\Na^{\sE_M} : \Ga^\infty(M,E_M) \to \Ga^\infty(M,E_M \ot T^* M).
\]
We apply the notation $c_M : \Cl(T M) \to \End( E_M)$ for the Clifford action whereas the module of smooth sections of the spinor bundle and the Hilbert space of $L^2$-spinors are denoted by
\[
\sE_M = \Ga^\infty(M,E_M) \q \Tex{and} \q L^2(\sE_M) = L^2(M,E_M),
\]
respectively. When $M$ is even-dimensional we denote the $\zz/2\zz$-grading operator by $\ga_M : E_M \to E_M$ and remark that this grading operator induces a $\zz/2\zz$-grading operator $\ga_M : L^2(\sE_M) \to L^2(\sE_M)$.
\medskip

We assume that the action $\al : \T^n \ti M \to M$ lifts to an action $\ov{\al} : \T^n \ti E_M \to E_M$ inducing unitary bundle maps
\[
\ov{\al}_t : E_M \to \al_t^* E_M \q t \in \T^n
\]
such that
\begin{itemize}
\item $\ov{\al}_t \ci c_M(X) = c_M( d\al_t(X)) \ci \ov{\al}_t$ and
\item $\ov{\al}_t \ci \Na^{\sE_M}_X = (\al_t^* \Na^{\sE_M})_X \ci \ov{\al}_t$,
\end{itemize}
for all smooth vector fields $X : M \to TM$ and all $t \in \T^n$. Remark here that 
\[
\al_t^* \Na^{\sE_M} : \Ga^\infty(M,\al_t^* E_M) \to \Ga^\infty(M, \al_t^* E_M \ot T^* M)
\]
denotes the pullback connection. In the even dimensional case, each $\ov{\al}_t$ is assumed to be even with respect to the grading on $E_M$. When these properties are satisfied we say that the action $\al : \T^n \ti M \to M$ \emph{admits a spin$^c$ lift} or that the \emph{spin$^c$ structure on $M$ is $\T^n$-equivariant}.
\medskip

Let us fix a skew-symmetric matrix $\te \in M_n(\rr)$.

For each $k \in \zz^n$ we define the spectral subspace
\[
C^\infty(M)_k = \big\{ f \in C^\infty(M) \mid \al_t(f) = t^k \cd f \, , \, \, \Tex{for all } t \in \T^n \big\}.
\]

\begin{defn}\label{defn:toric-def}
The coordinate algebra for the theta-deformed manifold is the unital $*$-subalgebra
\[
\C[M_\te] = \Tex{span}_{\C}\big\{ f \ot U^k \mid k \in \zz^n \, , \, \, f \in C^\infty(M)_k \big\} \su C^\infty(M) \ot \C[\T_\te^n].
\]
The unital $C^*$-algebra $C(M_\theta)$ is defined as the completion of $\C[M_\te]$ with respect to the supremum-norm
\[
\| x \|_{C(M_\te)} = \sup_{p \in M}\| x(p) \|_{C(\T_\te^n)} \q x \in \C[M_\te].
\]
The smooth functions on the theta-deformed manifold $M_\te$ is the unital Fr\'echet $*$-algebra $C^\infty(M_\te)$ obtained as the closure of the coordinate algebra $\C[M_\te]$ inside the projective tensor product of Fr\'echet $*$-algebras $C^\infty(M) \hot C^\infty(\T^n_\te)$.
%
\end{defn}

The lift $\ov{\al} : \T^n \ti E_M \to E_M$ induces a unitary representation of $\T^n$ on the Hilbert space of $L^2$-spinors:
\[
\wit \al : \T^n \to \U( L^2(\sE_M)) \q \wit \al_t(\xi) := \ov \al_t \ci \xi \ci \al_t^{-1}
\]
and for each $k \in \zz^n$, we thus have the spectral subspace
\[
L^2(\sE_M)_k = \big\{ \xi \in L^2(\sE_M) \mid \wit \al_t(\xi) = t^k \cd \xi \, , \, \, \Tex{for all } t \in \T^n \big\}. 
\]
In particular, we may define the isometry
\[
\begin{split}
& V : L^2(\sE_M) \to L^2(\sE_M) \hot L^2(\T^n_\te) \\
& V( \sum_{k \in \zz^n} \xi_k) = \sum_{k \in \zz^n} \xi_k \ot U^k.
\end{split}
\]
The left action of $C(M_\te)$ on $L^2(\sE_M)$ is then given by the representation
\[
\pi : C(M_\te) \to L^2(\sE_M) \q \pi(x) = V^* x V ,
\]
where we are suppressing the action of $C(M_\te)$ on the Hilbert space tensor product $L^2(\sE_M) \hot L^2(\T^n_\te)$ coming from the tensor product of the action of $C(M)$ on $L^2(\sE_M)$ and the action of $C(\T^n_\te)$ on $L^2(\T^n_\te)$.
\medskip

We denote the Dirac operator by $D_M : \dom(D_M) \to L^2(\sE_M)$ and recall that $D_M$ is the closure of the first order differential operator $\sD_M : \sE_M \to L^2(\sE_M)$ defined locally by the formula 
\[
\sD_M = i \sum_{j = 1}^{\dim(M)} c_M( e_j) \Na^{\sE_M}_{e_j} , 
\]
for any local orthonormal frame $\{e_j\}$ of the real tangent bundle. We remark that the Dirac operator $D_M$ satisfies the equivariance condition
\[
D_M \wit{\al}_t = \wit{\al}_t D_M \q \Tex{for all } t \in \T^n.
\]

We record the following result from \cite{CL01, CD02}:

%

\begin{thm}[Connes--Landi, Connes--Dubois-Violette]
Let $M$ be a compact Riemannian spin$^c$ manifold equipped with an isometric torus action admitting a spin$^c$ lift. Then the triple
\[
(C^\infty(M_\theta),L^2(\sE_M),D_M)
\]
is a spectral triple for the $C^*$-algebra $C(M_\theta)$ of the same parity as the dimension of $M$ and with grading operator $\ga_M : L^2(\sE_M) \to L^2(\sE_M)$ in the even dimensional case.
\end{thm}


We aim for a factorization of the theta-deformed spectral triple $(C^\infty(M_\theta),L^2(\sE_M),D_M)$ in terms of vertical and horizontal unbounded cycles, moreover this factorization result will be an unbounded analogue of the interior Kasparov product in bivariant K-theory. We will make heavy use of the results of \cite{KS17a} and in particular of Examples 24 and 27 therein. We restrict ourselves to the case where both the compact spin$^c$ manifold $M$ and the torus $\T^n$ are \emph{even} dimensional, but note that the remaining three cases can be treated by a similar argument.
\medskip

Our assumptions are as follows:

\begin{ass}\label{a:almreg}
We assume that our isometric action $\al : \T^{2m} \ti M \to M$ on the even dimensional compact spin$^c$ manifold $M$ is effective and admits a spin$^c$ lift. Letting $M_0 \subset M$ denote the principal stratum for the action of $\T^{2m}$ on $M$ we moreover assume that
\begin{itemize}
\item There exists a finite number $P_1,P_2,\ldots,P_l \subset M$ of compact embedded submanifolds, each without boundary and of codimension strictly greater than 1, such that
\[
M \sem M_0 = \cup_{j = 1}^l P_j.
\]
\end{itemize}
In particular, the quadruple $(M, M \sem M_0, M_0 / \T^{2m},q)$ is an almost-regular fibration of spin$^c$ manifolds in the sense of \cite[Definition 26]{KS17b}.
\end{ass}

The compact embedded submanifolds $P_j \su M$ are typically connected components of $H$-fixed points for some isotropy group $H \subset \T^{2m}$. Remark however that condition $(2)$ in the above assumption is \emph{not automatic} since we could have isotropy groups $H \neq \{0\}$ with $\T^{2m}/H$ of the same dimension as $\T^{2m}$.  
\medskip

The open dense submanifold $M_0 \subset M$ inherits a $\T^{2m}$-equivariant spin$^c$ structure from $M$ and we indicate with an extra zero subscript that spinor bundles, Clifford actions and connections are restricted to $M_0$.

Since the spin$^c$ structure on $M_0$ is equivariant we may also provide the quotient manifold 
\[
B = M_0/\T^{2m}
\]
with a spin$^c$ structure in such a way that the quotient map $q : M_0 \to M_0/\T^{2m}$ is a Riemannian submersion.

We denote the $\zz/2\zz$-graded spinor bundle by $E_B \to B$, the Clifford action by
\[
c_B : \Cl(TB) \to \End(E_B)
\]
and a fixed hermitian Clifford connection by
\[
\Na^{\sE_B} : \Ga^\infty_c(B,E_B) \to \Ga^\infty_c(B,E_B \ot T^* B).
\]
The associated Dirac operator is denoted by
\[
\sD_B : \Ga^\infty_c(B,E_B) \to L^2(B,E_B)
\]
and the closure by $D_B : \dom(D_B) \to L^2(B,E_B)$. We remark that $D_B$ need not be selfadjoint (since the principal orbit space in general fails to be complete). The triple $(C_c^\infty(B),L^2(B,E_B),D_B)$ is however still an even half-closed chain, representing the fundamental class of $B$ in $KK_0(C_0(B),\C)$. The grading operator $\ga_B : L^2(B,E_B) \to L^2(B,E_B)$ is induced by the grading operator on $E_B$. The even half-closed chain $(C_c^\infty(B),L^2(B,E_B),D_B)$ describes the horizontal part of our geometry.
%
%
\medskip

The $\theta$-deformation of $C_0(M_0)$ is the non-unital $C^*$-algebra $C_0((M_0)_\te)$ obtained as the $C^*$-norm closure of the $*$-subalgebra
\[
\C_c[ (M_0)_\te] = \Tex{span}_{\C} \big\{ f \ot U^k \mid k \in \zz^{2m} \, , \, \, f \in C_c^\infty(M_0)_k \big\} \subset C(M_\te), 
\]
where the spectral subspaces $C_c^\infty(M_0)_k$, $k \in \zz^{2m}$, are defined using the torus action on $M_0$. We also have the non-unital Fr\'echet $*$-algebra of theta-deformed smooth functions vanishing at infinity on the principal stratum. This Fr\'echet $*$-algebra is denoted by $C_0^\infty( (M_0)_\te)$ and is obtained as the closure of $\C_c[ (M_0)_\te]$ inside $C^\infty(M_\te)$.
\medskip

For the vertical part of our geometry, one proceeds mainly as in \cite{KS16,KS17b}, but we keep track of torus-actions and replace left actions by their theta-deformed analogues. 

We define a $\zz/2\zz$-graded smooth hermitian vector bundle of vertical spinors 
\[
E_V = \Hom( q^* E_B, M_0 \ti \C) \ot_{\Cl(T_H M_0)} E_{M_0} \to M_0,
\]
where $\Cl(T_H M_0) \to M_0$ denotes the Clifford bundle generated by the horizontal tangent vectors, thus vectors in the fiber of the horizontal tangent bundle 
\[
T_H M_0 := \Tex{Ker}(d q)^\perp = (T_V M_0)^\perp.
\]
Remark that the grading operator on the bundle $E_V \to M_0$ is given by 
\[
\ga_V := (q^* \ga_B)^\da \ot \ga_{M_0}
\]
and that the hermitian form is explained in \cite[Definition 7]{KS16}, (the $\da$ in $(q^* \ga_B)^\da$ refers to the operation on the dual bundle given by precomposition). Moreover, $E_V \to M_0$ carries a Clifford action
\[
c_V : \Cl(T_V M_0) \to \End( E_V) \q c_V( \xi) = (q^* \ga_B)^\da \ot c_{M_0}(\xi)
\]
by vertical tangent vectors and admits an even hermitian Clifford connection
\[
\Na^{\sE_V} : \Ga^\infty_c(M_0,E_V) \to \Ga^\infty_c(M_0,E_V \ot T^* M_0),
\]
given by the explicit formula in \cite[Proposition 11]{KS16}. The manifold $E_V$ also carries a smooth action of $\T^{2m}$,
\[
\ov{\al}_V : \T^{2m} \ti E_V \to E_V \q (\ov{\al}_V)_t := 1 \ot \ov{\al}_t 
\]
and it can be verified that this action is compatible with the above data in the sense that
\begin{itemize}
\item $\ov{\al}_V$ induces even unitary bundle maps $(\ov{\al}_V)_t : E_V \to \al_t^*(E_V)$;
\item $(\ov{\al}_V)_t \ci c_V(Z) = c_V( d\al_t(Z)) \ci (\ov{\al}_V)_t$ and
\item $(\ov{\al}_V)_t \ci \Na^{\sE_V}_X = (\al_t^* \Na^{\sE_V})_X \ci (\ov{\al}_V)_t$,
\end{itemize}
for all $t \in \T^{2m}$ and all smooth vector fields $X,Z : M_0 \to T M_0$ with $Z$ vertical. Remark in this respect that $d \al_t(Y_H) = Y_H \ci \al_t$ whenever $Y_H = (dq)^*( Y \ci q)$ is a horizontal lift of a vector field $Y$ on $B$ (where the $*$ in $(dq)^*$ refers to the adjoint operation).

We denote the right $C_c^\infty(B)$-module of smooth compactly supported sections of $E_V$ by
\[
\sE_V^c = \Ga^\infty_c(M_0,E_V),
\]
where the right action is defined via the pullback $q^* : C^\infty_c(B) \to C_c^\infty(M_0)$. Using integration along the orbits of $\T^{2m}$, we may define a $C_0(B)$-valued inner product $\inn{\cd,\cd}_X$ on $\sE_V^c$ and the completion is a $\zz/2\zz$-graded Hilbert $C^*$-module $X$ over $C_0(B)$. The grading operator 
\[
\ga_X = \ga_V : X \to X  
\]
is induced by the grading operator on the bundle $E_V$. For each $t \in \T^{2m}$ we define the \emph{vertical spin$^c$ lift}
\[
(\wit{\al}_V)_t : \sE_V^c \to \sE_V^c \q (\wit{\al}_V)_t(\xi) = (\ov{\al}_V)_t \ci \xi \ci \al_t^{-1}
\]
and notice that $\wit{\al}_V$ extends to a strictly continuous even unitary representation
\[
\wit{\al}_V : \T^{2m} \to \U(X)
\]
on the Hilbert $C^*$-module completion $X$.

For each $k \in \zz^{2m}$, we thus have the spectral submodule
\[
X_k = \big\{ \xi \in X \mid \wit{\al}_t(\xi) = t^k \cd \xi \big\} \subset X
\]
and the associated even bounded adjointable isometry
\[
V : X \to X \hot L^2(\T^{2m}_\te) \q V( \sum_{k \in \zz^{2m}} \xi_k) = \sum_{k \in \zz^{2m}} \xi_k \ot U^k.
\]
We may thus turn $X$ into a $\zz/2\zz$-graded $C^*$-correspondence from $C_0\big( (M_0)_\te \big)$ to $C_0(B)$ by defining the left action via the $*$-homomorphism
\[
\pi : C_0\big( (M_0)_\te\big) \to \sL( X ) \q \pi(x) = V^* x V,
\]
where we are suppressing the action of $C_0\big( (M_0)_\te \big) \subset C_0(M_0) \hot C(\T^{2m}_\te)$ on the Hilbert $C^*$-module $X \hot L^2(\T^{2m}_\te)$ coming from the tensor product of the action of $C_0(M_0)$ on $X$ (via pointwise multiplication) and the action of $C(\T^{2m}_\te)$ on $L^2(\T^{2m}_\te)$ discussed in Subsection \ref{ss:nc-torus}.

The vertical Dirac operator $D_V : \dom(D_V) \to X$ is then an unbounded operator on the Hilbert $C^*$-module $X$. This unbounded operator is defined as the closure of the first order differential operator $\sD_V : \sE_V^c \to X$ given locally by the expression
\[
\sD_V = i \sum_{j = 1}^{2m} c_V(e_j) \Na^{\sE_V}_{e_j}
\]
for any local orthonormal frame $\{e_j\}$ of the real vertical tangent bundle. We record that the vertical Dirac operator is equivariant in the sense that
\[
D_V (\wit{\al}_V)_t = (\wit{\al}_V)_t D_V \q \Tex{for all } t\in \T^{2m}.
\]

The vertical part of our theta-deformed geometric data can then be summarized in the following:

\begin{prop}\label{p:vertthet}
The triple $(C_0^\infty( (M_0)_\te), X,D_V)$ is an even unbounded Kasparov module from $C_0( (M_0)_\te)$ to $C_0(B)$.
\end{prop}
\begin{proof}
It follows from \cite[Proposition 13]{KS17b} that $D_V : \dom(D_V) \to X$ is an odd selfadjoint and regular unbounded operator and that $m(f) \cd (i + D_V)^{-1} : X \to X$ is a compact operator for all $f \in C_0(M_0)$, where the bounded adjointable operator $m(f) : X \to X$ is induced by the pointwise multiplication with the function $f$. 

To show that $\pi(x) \cd (i + D_V)^{-1} : X \to X$ is a compact operator for all $x \in C_0\big( (M_0)_\te\big)$, we may restrict to the case where $x = f \ot U^k$ for some $k \in \zz^{2m}$ and some compactly supported $f \in C_c^\infty(M_0)_k$ in the $k^{\Tex{th}}$ spectral subspace. In fact, since we may find $g \in C_0(B)$ with $f \cd (g \ci q)$, we may even assume that $k = (0,0,\ldots,0)$. But in this case, we have that $\pi(x) = m(f)$ and we are done.

Let now $x \in C_0^\infty( (M_0)_\te)$. The fact that $\pi(x)$ preserves the domain of $D_V$ and that the commutator $[D_V,\pi(x)] : \dom(D_V) \to X$ has a bounded extension to $X$ follows exactly as in the proof of Proposition \ref{p:vertsphe}, using the $\T^{2m}$-equivariance of the vertical Dirac operator $D_V$.
\end{proof}


In order to lift the Dirac operator on the base $D_B : \dom(D_B) \to L^2(B,E_B)$ to the interior tensor product $X \hot_{C_0(B)} L^2(B,E_B)$, we need an even equivariant metric connection for the $C^\infty_c(B)$-valued inner product $\langle\cdot,\cdot \rangle_X$ on $\sE^c_V$. To this end, we modify the hermitian Clifford connection on $E_V$ by the mean curvature
\[
k : T_H M_0 \to M_0 \ti \C,
\]
which is defined as the trace of the second fundamental form, see \cite{KS16,BGV92} for the details. Our even metric connection is then given by
\[
\nabla^X_Z(\xi) = \nabla^{\sE_V}_{Z_H}(\xi) + \frac{1}{2} k(Z_H) \cdot \xi,
\]
for any smooth vector field $Z$ on $B$, with horizontal lift $Z_H : M_0 \to T_H M_0$ and any $\xi \in \sE_V^c \subseteq X$, see \cite[Definition 18]{KS16}. Since our smooth action $\al : \T^{2m} \ti M_0 \to M_0$ is isometric, it can be verified that the mean curvature $k(Z_H)$ is invariant under this action and hence defines an element in $C^\infty(B)$ for every horizontal lift $Z_H$. Moreover, it follows from the properties of the lift $\ov{\al}_V$ that $\Na^{\sE_V}_{Z_H} (\wit{\al}_V)_t = (\wit{\al}_V)_t \Na^{\sE_V}_{Z_H}$ for every $t \in \T^{2m}$ and every horizontal lift $Z_H$. We thus conclude that our metric connection is equivariant with respect to the vertical spin$^c$ lift:
\[
\Na^X_Z (\wit{\al}_V)_t = (\wit{\al}_V)_t \Na^X_Z : \sE_V^c \to X \q t \in \T^{2m} \, , \, \, Z \in \Ga^\infty(B,TB).
\]

%
%
%

The tensor sum we are after is given by the symmetric unbounded operator 
\[
\begin{split}
\sD_V \ti_{\Na} \sD_B & = \sD_V \otimes 1  + (\ga_X \ot 1) (1 \otimes_\nabla \sD_B)  \\ 
& \qq : \sE_V^c \otimes_{C^\infty_c(B)} \Ga^\infty_c(B,E_B) \to X \hot_{C_0(B)} L^2(B,E_B) \, ,
\end{split}
\]
The closure of the symmetric unbounded operator $\sD_V \ti_{\Na} \sD_B$ will be denoted by $D_V \times_\nabla D_B$.

The vertical spin$^c$ lift $\wit \al_V$ induces a unitary representation
\[
\wit \al_V \ot 1 : \T^{2m} \to \U( X \hot_{C_0(B)} L^2(B,E_B))
\]
and since both $D_V$ and the even metric connection $\Na^X$ are equivariant for the vertical spin$^c$ lift, we conclude that our tensor sum is equivariant as well
\[
(D_V \ti_{\Na} D_B)\big(  (\wit{\al}_V)_t \ot 1 \big) = \big(  (\wit{\al}_V)_t \ot 1 \big) (D_V \ti_{\Na} D_B) \q t \in \T^{2m}.
\]

\begin{thm}\label{thm:tensor-sum-toric}
Up to unitary equivalence of $C^*$-correspondences (from $C_0( (M_0)_\te)$ to $\C$), we have the equality of selfadjoint operators 
\[
D_V \ti_{\Na} D_B = \overline {\sD_{M_0} - \frac{i}{8} c_{M_0}(\Omega)},
\]
where $\sD_{M_0} : \Ga^\infty_c(M_0,E_{M_0}) \to L^2(M_0,E_{M_0})$ is the Dirac operator and $c_{M_0}(\Omega) : \Ga^\infty_c(M_0, E_{M_0}) \to L^2(M_0, E_{M_0})$ denotes Clifford multiplication ({\it cf.} \cite[Section 3.3]{KS17a} for the precise formula) by the curvature 2-form $\Omega : \La^2(T_H M_0) \ot T_V M_0 \to M_0 \ti \C$ defined by
\[
\Om(X,Y,Z) = \inn{ [X,Y],Z}_{M_0},
\]
for all real horizontal vector fields $X,Y$ and every real vertical vector field $Z$.
\end{thm}
\begin{proof}
This follows immediately from \cite[Proposition 18]{KS17b} provided that the unitary isomorphism
\[
W \Ga : X \hot_{C_0(B)} L^2(B,E_B) \to L^2(M_0,E_{M_0})
\]
(of $C^*$-correspondences from $C_0(M_0)$ to $\C$) appearing in that proposition is torus equivariant, so that $W \Ga$ also intertwines the left actions of the theta-deformation $C_0( (M_0)_\te)$. We recall in this respect that 
\[
\Ga = (\ga_X \ot 1) \frac{1 + \ga_X \ot \ga_B}{2} + \frac{1 - \ga_X \ot \ga_B}{2}
\]
and that $W$ is given by the formula
\[
W : (\bra{\xi} \ot s ) \ot r \mapsto (\ket{r \ci q}\bra{\xi} )(s),
\]
for all $\xi \in \Ga^\infty(M_0,q^* E_B)$, $s \in \Ga^\infty_c(M_0, E_{M_0})$ and $r \in \Ga^\infty_c(B,E_B)$, see \cite[Proposition 14]{KS17b}. The torus equivariance therefore reduces to showing that
\[
c_{M_0}( Y_H) {\wit \al}_t(s) = {\wit \al}_t c_{M_0}( Y_H) (s),
\]
whenever $t \in \T^{2m}$ and $Y_H : M_0 \to T_H M_0$ is the horizontal lift of a smooth vector field $Y$ on $B$. But this follows since the unitary operator ${\wit \al}_t : L^2(M_0,E_{M_0}) \to L^2(M_0,E_{M_0})$ comes from the lift $\ov \al_t : E_M \to \al_t^* E_M$ and since $d\al_t(Y_H) = Y_H \ci \al_t$ for horizontal lifts. 
\end{proof}

Again, using the embedding map $\imath: M_0 \to M$ we can put this result in the context of spectral triples and KK-theory. We first remark that, up to unitary equivalence, we have the identity
\begin{equation}\label{eq:restricII}
\io^*( C^\infty(M_\theta), L^2(\sE_M), D_M ) = ( C_0^\infty((M_0)_\theta), L^2(M_0,E_{M_0}), D_{M_0}).
\end{equation}
Indeed, the second condition in Assumption \ref{a:almreg} ensures that \newline $D_M : \dom(D_M) \to L^2(\sE_M)$ and $D_{M_0} : \dom(D_{M_0}) \to L^2(M_0,E_{M_0})$ both have the dense subspace $\Ga_c^\infty(M_0,E_{M_0}) \subset L^2(M,E_M)$ as a core (upon identifying $L^2(M_0,E_{M_0})$ and $L^2(M,E_M)$ via the unitary isomorphism given by extension by zero), see \cite[Proposition 30]{KS17b} for a full proof. Notice also that the pullback along the inclusion $C_0^\infty((M_0)_\theta) \subset C^\infty(M_\theta)$ only changes the left action in question.

We shall now see how our tensor sum decomposition of $D_M$ is related to the interior Kasparov product in bivariant K-theory. Let us denote the classes in even $KK$-theory associated to the unbounded Kasparov module $(C_0^\infty( (M_0)_\te), X,D_V)$, the half-closed chain \newline $(C^\infty_c(B), L^2(B,E_B), D_B)$ and the spectral triple $(C^\infty(M_\te) , L^2(\sE_M), D_M)$ by 
\[
\begin{split}
& [(D_V)_\te] \in KK_0\big(C_0( (M_0)_\te), C_0(B)\big) \\
& [D_B] \in KK_0( C_0(B),\C) \q \Tex{and} \q [D_{M_\te}] \in KK_0( C(M_\te),\C),
\end{split}
\]
respectively.
\medskip

We begin with a preliminary lemma relating our tensor sum $D_V \ti_{\Na} D_B$ to the Dirac operator $D_M$ at the level of $KK$-theory.

\begin{lma}\label{l:tensorpull}
The triple $\big( C_0^\infty( (M_0)_\te) , L^2(M_0,E_{M_0}), D_V \ti_\Na D_B\big)$ is an even half-closed chain from $C_0( (M_0)_\te)$ to $\C$ and we have the identity
\[
[ (D_V \ti_\Na D_B)_\te] = \io^*[ D_{M_\te}]
\]
for the associated class $[ (D_V \ti_\Na D_B)_\te]$ in the $KK$-group \newline $KK_0( C_0((M_0)_\te),\C)$.
\end{lma}
\begin{proof}
To ease the notation, we define
\[
\sD_{M_0}' := \sD_{M_0} - \frac{i}{8} c_{M_0}(\Om) : \Ga^\infty_c(M_0, E_{M_0}) \to L^2(M_0,E_{M_0})
\]
and let $D_{M_0}' : \dom( D_{M_0}') \to L^2(M_0,E_{M_0})$ denote the closure.

By Theorem \ref{thm:tensor-sum-toric}, we may (in the present lemma) replace the triple 
\[
\big( C_0^\infty( (M_0)_\te) , L^2(M_0,E_{M_0}), D_V \ti_\Na D_B\big)
\]
by the triple
\[
\big( C_0^\infty( (M_0)_\te), L^2( M_0, E_{M_0}), D_{M_0}' \big). 
\]

Now, since the unbounded operator 
\[
\sD_{M_0}' : \Ga_c^\infty( M_0, E_{M_0}) \to L^2( M_0, E_{M_0})
\]
is an odd symmetric and elliptic first order differential operator we know from \cite{BDT89,Hil10} that the non-deformed triple
\[
\big( C_c^\infty(M_0), L^2( M_0, E_{M_0}), D_{M_0}' \big)
\]
is an even half-closed chain. The fact that the theta-deformed triple
\[
\big( C_0^\infty( (M_0)_\te), L^2( M_0, E_{M_0}), D_{M_0}' \big)
\]
is an even half-closed chain now follows from the argument given in the proof of Proposition \ref{p:vertthet}. Indeed, since both $\sD_{M_0}$ and $\sD_V \ti_{\Na} \sD_B$ are equivariant for the respective torus actions and the unitary equivalence $W \Ga : X \hot L^2(B,E_B) \to L^2(M_0,E_{M_0})$ (see the proof of Theorem \ref{thm:tensor-sum-toric}) is also torus equivariant, we conclude that
\[
D_{M_0}': \dom( D_{M_0}') \to L^2(M,E_M)
\]
must be torus equivariant as well. Notice that the extra detail that $C_0^\infty( (M_0)_\te)$ must map the domain of the adjoint $(D_{M_0}')^*$ into the domain of $D_{M_0}'$ follows by factorizing an arbitrary element $x \in \C_c[ (M_0)_\te]$ as a product $x \cd ((f \ci q) \ot 1)$ where $f \in C_c^\infty(B)$ and then use a continuity argument to pass to general elements in $C_0^\infty((M_0)_\te)$.

To show that the even half-closed chains
\begin{equation}\label{eq:nonlocal}
\begin{split}
& \big( C_0^\infty( (M_0)_\te), L^2( M_0, E_{M_0}), D_{M_0}' \big) \q \Tex{and} \\
& \big( C_0^\infty( (M_0)_\te), L^2( M_0, E_{M_0}), D_{M_0} \big)
\end{split}
\end{equation}
represent the same class in $KK_0( C_0( (M_0)_\theta),\C)$ we follow the argument given in the proof of \cite[Lemma 19]{KS17b}. Indeed, we choose a positive function $f \in C_0(B)$ such that $x := (f \ci q) \otimes 1 \in C_0( (M_0)_\theta)$ satisfies that 
\begin{itemize}
\item $x \cd C_0( (M_0)_\theta) \subset C_0( (M_0)_\theta)$ is norm-dense;
\item $\pi(x)$ preserves the core $\Ga^\infty_c(M_0,E_{M_0})$ and the commutators
\[
\begin{split}
& [\sD_{M_0}, \pi(x)] : \Ga^\infty_c(M_0,E_{M_0}) \to L^2(M_0,E_{M_0}) \q \Tex{and} \\
& [\sD_{M_0}', \pi(x)] : \Ga^\infty_c(M_0,E_{M_0}) \to L^2(M_0,E_{M_0}) 
\end{split}
\]
have bounded extensions to $L^2(M_0,E_{M_0})$;
\item The (a priori) unbounded operator $\ov{c_{M_0}(\Om)} \pi(x)$ is defined everywhere on $L^2(M_0,E_{M_0})$ and is in fact bounded.
\end{itemize}
We then have the localized unbounded operators $\pi(x) D_{M_0} \pi(x)$ and $\pi(x) \sD_{M_0}' \pi(x) = \pi(x) \sD_{M_0} \pi(x) - \frac{i}{8} \pi(x) c_{M_0}(\Om) \pi(x)$ both with domain $\Ga_c^\infty(M_0,E_{M_0}) \subset L^2(M_0,E_{M_0})$. By \cite[Theorem 13 and Theorem 19]{KS17a} these localized unbounded operators are essentially selfadjoint and the even spectral triples
\begin{equation}\label{eq:localize}
\begin{split}
& \big( C^\infty_0( (M_0)_\te), L^2( M_0, E_{M_0}), \ov{ \pi(x) \sD_{M_0}' \pi(x)} \big) \q \Tex{and} \\
& \big( C^\infty_0( (M_0)_\te), L^2( M_0, E_{M_0}), \ov{ \pi(x) \sD_{M_0} \pi(x)} \big)
\end{split}
\end{equation}
represent the same classes in $KK_0( C_0( (M_0)_\te), \C)$ as our original half-closed chains (from \eqref{eq:nonlocal}). Remark that our localized unbounded operators really define spectral triples, so that unbounded modular cycles are not needed here, since the modular operator $\pi(x^2)$ commutes with all elements in the Fr\'echet $*$-algebra $C_0^\infty( (M_0)_\te)$. The result of the lemma now follows by noting that the two spectral triples in Equation \eqref{eq:localize} represent the same class in $KK_0( C_0( (M_0)_\te), \C)$ since $\ov{\pi(x) \sD_{M_0}' \pi(x)}$ is a bounded perturbation of $\ov{\pi(x) \sD_{M_0} \pi(x)}$.
\end{proof}

\begin{thm}
\label{thm:fact-toric}
We have the identity
\[
\io^*[ D_{M_\te}] = [ (D_V)_\te] \hot_{C_0(B)} [D_B]
\]
in the $KK$-group $KK_0( C_0((M_0)_\te),\C)$.
\end{thm}
\begin{proof}
We follow the proof of \cite[Theorem 22]{KS17b} closely. In particular, we rely on our version of Kucerovsky's theorem for half-closed chains established in \cite{KS17a}.

By Lemma \ref{l:tensorpull}, we may replace $\io^*[D_{M_\te}]$ with $[ (D_V \ti_{\Na} D_B)_\te]$ and, by \cite[Theorem 34]{KS17a}, it then suffices to verify that our three triples
\[
\begin{split}
& (C_0^\infty( (M_0)_\te), X,D_V) \q ( C^\infty_c(B), L^2(B,E_B), D_B) \q \Tex{and} \\
& \big( C_0^\infty( (M_0)_\te) , L^2(M_0,E_{M_0}), D_V \ti_\Na D_B\big)
\end{split}
\]
satisfy the connection condition and the local positivity condition given in \cite[Definition 27 and Definition 29]{KS17a}.

The connection condition is proved exactly as in \cite[Theorem 22]{KS17b}. The local positivity condition requires us to choose a localizing subset, see \cite[Definition 28]{KS17a}. As in the proof of \cite[Theorem 22]{KS17b} we choose a countable smooth partition of unity $\{ \chi_m\}$ for the base manifold $B$ such that each $\chi_m : B \to [0,1]$ has compact support. We then define the localizing subset 
\[
\La = \big\{ \chi_m \ci q \ot 1 \mid m \in \nn \big\} \subset \C_c[ (M_0)_\te].
\]
Since the left action of each $\chi_m \ci q \ot 1$ on $X$ is simply given by the multiplication operator with the compactly supported smooth function $\chi_m \ci q : M_0 \to [0,1]$, we obtain the local positivity condition using the argument given in the proof of \cite[Theorem 22]{KS17b}.
\end{proof}


\newcommand{\noopsort}[1]{}\def\cprime{$'$}
\providecommand{\bysame}{\leavevmode\hbox to3em{\hrulefill}\thinspace}
\providecommand{\MR}{\relax\ifhmode\unskip\space\fi MR }
\providecommand{\MRhref}[2]{%
  \href{http://www.ams.org/mathscinet-getitem?mr=#1}{#2}
}
\providecommand{\href}[2]{#2}

\end{document}